\documentclass[a4paper]{scrartcl}
\usepackage[utf8]{inputenc}
\usepackage{amsmath,amssymb,amsthm,amsfonts}
\usepackage{mathtools,thmtools}
\usepackage[british]{babel}
\usepackage{csquotes}
\usepackage{hyperref}
\usepackage[nameinlink,noabbrev]{cleveref}
\usepackage{enumerate}
\usepackage{xcolor}

\usepackage[maxbibnames=5,sorting=nyt,sortcites,giveninits=true]{biblatex}
\DeclareNameAlias{sortname}{family-given}
\newcommand{\R}     {\mathbb{R}}
\newcommand{\C}     {\mathbb{C}}

\newcommand{\Z}     {\mathbb{Z}}
\newcommand{\T}     {\mathbb{T}}
\newcommand{\eps}   {\varepsilon}
\let\S\relax
\newcommand{\S}     {\mathcal{S}}
\newcommand{\V}     {\mathcal{V}}
\newcommand{\Ecal}  {\mathcal{E}}
\newcommand{\diff}  {\mathop{}\!\mathrm{d}}
\newcommand{\OU}    {\mathrm{OU}}
\newcommand{\loc}   {\mathrm{loc}}

\newcommand{\muq}   {\overline\mu}
\newcommand{\muqk}  {\muq\otimes\kappa}
\newcommand{\muk}   {\mu\otimes\kappa}
\newcommand{\stpi}  {\mathrm{STPI}}
\newcommand{\rel}   {\mathrm{rel}}

\let\Re\relax

\DeclareMathOperator{\Re}   {Re}

\DeclareMathOperator{\tr}   {tr}
\DeclareMathOperator{\gap}  {gap}
\DeclareMathOperator{\spec} {spec}

\DeclareMathOperator{\Ex}   {\mathbb{E}}

\DeclareMathOperator{\unif} {Unif}
\DeclareMathOperator{\dom}  {Dom}
\DeclareMathOperator{\sing} {s}
\DeclareMathOperator{\divg} {div}
\DeclareMathOperator{\grad} {grad}

\DeclarePairedDelimiterX{\norm}[1]{\lVert}{\rVert}{#1}

\theoremstyle{plain}
\newtheorem{theo}{Theorem}
\newtheorem{lemm}[theo]{Lemma}
\newtheorem{coro}[theo]{Corollary}

\theoremstyle{definition}
\newtheorem{defi}[theo]{Definition}
\newtheorem{exam}[theo]{Example}
\newtheorem{rema}[theo]{Remark}
\newtheorem{rema*}{Remark}
\newtheorem{assu}{Assumption}
\crefname{lemm}{lemma}{lemmas}
\crefname{theo}{theorem}{theorems}

\title{Non-reversible lifts of reversible diffusion processes and relaxation times}

\author{Andreas Eberle\thanks{E-Mail: \href{mailto:eberle@uni-bonn.de}{eberle@uni-bonn.de}, ORCID: \href{https://orcid.org/0000-0003-0346-3820}{0000-0003-0346-3820}}\qquad Francis Lörler\thanks{E-Mail: \href{mailto:loerler@uni-bonn.de}{loerler@uni-bonn.de}, ORCID: \href{https://orcid.org/0009-0007-3177-1093}{0009-0007-3177-1093}}\medskip\\Institute for Applied Mathematics, University of Bonn.\\ Endenicher Allee 60, 53115 Bonn, Germany.}

\begin{document}

\maketitle

\begin{abstract}
    We propose a new concept of lifts of reversible diffusion processes and show that various well-known non-reversible Markov processes arising in applications are lifts in this sense of simple reversible diffusions. Furthermore, we introduce a concept of non-asymptotic relaxation times and show that these can at most be reduced by a square root through lifting, generalising a related result in discrete time. Finally, we demonstrate how the recently developed approach to quantitative hypocoercivity based on space-time Poincaré inequalities can be rephrased and simplified in the language of lifts and how it can be applied to find optimal lifts.
    \par\vspace\baselineskip
\end{abstract}

\section{Introduction}
    Recently, there has been growing interest in accelerated convergence to stationarity of Markov processes due to non-reversibility \cite{Diaconis2000Lift,Chen1999Lift,Lelievre2013Optimal,Guillin2016Optimal,Eberle2019Langevin,Chen2013Accelerating,Hwang1993Accelerating,Hwang2005Accelerating}. In parallel, Markov chain Monte Carlo methods based on non-reversible Markov processes have become popular in applications \cite{Neal2011HMC,Krauth2021ECMC,Lei2018Mixing,Bierkens2017CurieWeiss,Bierkens2019ZigZag,Monmarche2021HighDimensional,Dalalyan2020KineticLangevin,Bouchard-Cote2018BPS}. Despite various important results on convergence of non-reversible Markov processes and diffusive-to-ballistic speed-ups in specific models \cite{Guillin2016Optimal,Lelievre2013Optimal,Eberle2019Langevin,Lu2022PDMP,Cao2019Langevin,Eberle2020Coupling}, a systematic understanding of these phenomena is still missing. Indeed, it is not even clear how crucial concepts such as lifts and relaxation times should be adequately defined for general Markov processes in continuous time. The goal of this paper is to propose a concept of non-reversible lifts for reversible diffusion processes and to show that many non-reversible processes considered recently in applications are actually lifts in this sense of very simple reversible diffusions. Moreover, we will introduce a concept of non-asymptotic relaxation times and show how lower and upper bounds on relaxation times of lifts can be derived.
    
    Let us first recall facts from the discrete time case. Here a concept of lifts of Markov chains has been developed by \cite{Diaconis2000Lift} and \cite{Chen1999Lift}, and a definition of asymptotic relaxation times of non-reversible Markov chains has very recently been given by \cite{Chat2023Spectral}. Suppose that $\S$ and $\V$ are measurable state spaces (usually interpreted as position and velocity), $\mu$ is a probability measure on $\S$ and $\hat\mu(\diff x\diff v)=\mu(\diff x)\kappa(x,\diff v)$ for some Markov kernel $\kappa$ from $\S$ to $\V$.
    In \cite{Chen1999Lift}, a Markov chain on $\S\times\V$ with transition kernel $\hat p$ is called a lift of a Markov chain on $\S$ with transition kernel $p$ and invariant measure $\mu$ if for any $x\in\S$ and any measurable subset $A\subseteq\S$
    \begin{equation*}
        \int_{\V}\hat p((x,v),A\times\V)\kappa(x,\diff v)=p(x,A)\,.
    \end{equation*}
    Equivalently, if $\pi\colon\S\times\V,\,(x,v)\mapsto x$ denotes the projection on the first coordinate, then for any bounded measurable function $f\colon\S\to\R$  and $x\in\S$
    \begin{equation*}
        \int_{\V}\hat p(f\circ\pi)(x,v)\kappa(x,\diff v)=(pf)(x)\,.
    \end{equation*}
    Note that the definition of a lift depends only on the one-step transition kernel, i.e.\ on the instantaneous change. For the $n$-step transition kernels, corresponding identities do not hold.
    
    For example, a lift of the symmetric random walk on the discrete circle $\S=\Z/(n\Z)$ is given by a Markov chain on the product space $\S\times\{+1,-1\}$ that moves deterministically in the $x$-coordinate with velocity $v=+1$ or $v=-1$, respectively, and at each step flips its velocity with a given probability $\eps\in[0,1]$. For $\eps$ of order $\Theta(\frac{1}{n})$, the mixing time of the lifted process is of order $\Theta(n)$ whereas the mixing time of the random walk is of order $\Theta(n^2)$ if $n$ is odd and infinite if $n$ is even \cite{Diaconis2000Lift}. This corresponds to a diffusive to ballistic speed-up of convergence to equilibrium. In general, based on a conductance argument, \cite{Chen1999Lift} have shown a lower bound on mixing times of lifted Markov chains, and for reversible Markov chains on finite graphs, they have shown the existence of ``optimal'' lifts with a mixing time that is at most a constant factor away from the lower bound. However, this construction relies on adding many edges corresponding to non-local interactions to the underlying graph. Such a construction is not very natural and it is not clear how this approach can be extended to the continuous time setting.
    
    As a motivating example for the continuous time case, let us consider Brownian motion on a circle $\T_a=\R/(a\Z)$ with circumference $a>0$. This is a reversible diffusion process, and natural candidates for lifts are piecewise deterministic Markov processes $(X_t,V_t)$ with state space $\T_a\times\R=\{(x,v)\colon x\in\T_a,v\in\R\}$ that move deterministically with velocity $v$ in the first coordinate and possibly update the velocity in the second coordinate after independent exponentially distributed waiting times to new values drawn from $\kappa=\mathcal{N}(0,1)$. These processes have invariant measure $\hat\mu=\mu\otimes\kappa$ where $\mu=\unif(\T_a)$. 
    
    Let $(P_t)_{t\geq 0}$ and $(\hat P_t)_{t\geq 0}$ denote the transition semigroups of Brownian motion and the lifted Markov process on $L^2(\T_a,\mu)$ and $L^2(\T_a\times\R,\mu\otimes\kappa)$, respectively, and suppose that $(L,\dom(L))$ and $(\hat L,\dom(\hat L))$ are the corresponding generators. Explicitly, $Lf(x)=\frac{1}{2}f''(x)$ and $\hat Lf(x,v)=v\partial_vf(x,v)+\lambda(\int_\R f(x,w)\kappa(\diff w)-f(x,v))$. It can be shown in various ways that for any function $f\in\dom(L)$ and $x\in\T_a$, in the limit $t\to0$,
    \begin{equation}\label{eq:lift1}
        \int_\R\hat P_t(f\circ\pi)(x,v)\kappa(\diff v)=(P_{t^2}f)(x)+o(t^2)\,,
    \end{equation}
    or equivalently 
    \begin{equation}\label{eq:lift2}
        \lim_{t\to0}\frac{1}{t^2}\int_\R(\hat P_t(f\circ\pi)-f\circ\pi)(\cdot,v)\kappa(\diff v)=\lim_{t\to 0}\frac{1}{t^2}(P_{t^2}f-f)=Lf\,.
    \end{equation}
    The quadratic and linear scalings in \eqref{eq:lift1} are characteristic of diffusive and ballistic motion and play a central role for the second order lifts we consider. 
    To verify \eqref{eq:lift1} and \eqref{eq:lift2}, note that if the jump intensity $\lambda$ is zero, then $X_t=x+vt\mod a$, where $x$ and $v$ are the initial position and velocity. If $v$ is drawn from $\kappa$, then $X_t$ has law $\mathcal{N}(x,t^2)$ and \eqref{eq:lift1} holds without the $o(t^2)$ error term. For $\lambda>0$, \eqref{eq:lift1} can also be verified by an explicit computation of $\hat P_t$ by conditioning on the jump times. Alternatively, \eqref{eq:lift1} follows by verifying that the time derivatives up to order two at $t=0$ agree on both sides. Indeed, $f\in\dom(L)$ implies $f\circ\pi\in\dom(\hat L)$, 
    \begin{align*}
        \int_\R\hat L(f\circ\pi)(x,v)\kappa(\diff v) &= \int_\R vf'(x)\kappa(\diff v) =0\, ,\qquad\text{and}\\
        \int_\R\hat L^2(f\circ\pi)(x,v)\kappa(\diff v) &= \int_\R(v^2f''(x)-\lambda vf'(x))\kappa(\diff v) = f''(x) = 2Lf(x)\,.
    \end{align*}
    For the applications we are interested in, it will be useful to rephrase \eqref{eq:lift1} and \eqref{eq:lift2} in terms of the infinitesimal generators. To this end, we observe that for $f\in L^2(\S,\mu)$,
    \begin{equation}
        \int_{\T_a\times\R}\hat P_t(f\circ\pi)^2\diff\hat\mu = \int_{\T_a\times\R}(f\circ\pi)^2\diff\hat\mu+o(t^2)\,.
    \end{equation}
    Indeed, this follows from a computation of the first two derivatives, noting that the noise acts only on the velocity and the generator of the deterministic part of the motion is antisymmetric. Therefore, \eqref{eq:lift2} implies that for any $f\in\dom(L)$ and $g\in L^2(\S,\mu)$,
    
    \begin{equation}\label{eq:lift3}
        \int_{\T_a\times\R}\hat L(f\circ\pi)g\circ\pi\diff\hat\mu=\lim_{t\to0}\frac{1}{t}\int_{\T_a\times\R}(\hat P_t(f\circ\pi)-f\circ\pi)g\circ\pi\diff\hat\mu=0\,,
    \end{equation}
    and
    \begin{align}\label{eq:lift4}
    \begin{split}
        \int_{\T_a\times\R}\hat L(f\circ\pi)^2\diff\hat\mu&=\lim_{t\to0}\frac{1}{t^2}\int_{\T_a\times\R}(\hat P_t(f\circ\pi)-f\circ\pi)^2\diff\hat\mu\\
        &=\lim_{t\to0}\frac{2}{t^2}\int_{\T_a\times\R}f\circ\pi(f\circ\pi-\hat P_t(f\circ\pi))\diff\hat\mu = -2\int_{\T_a}fLf\diff\mu\,.
    \end{split} 
    \end{align}
    The identities \eqref{eq:lift3} and \eqref{eq:lift4} motivate the general definition of second-order lifts in \Cref{def:lifts} below.
    
    In the following section we will give a formal definition of a second-order lift of a reversible diffusion process. We will then consider examples and show in particular that various well-known non-reversible Markov processes are second-order lifts of overdamped Langevin dynamics. In \Cref{sec:relaxation} we introduce a concept of non-asymptotic relaxation times for non-reversible Markov processes in continuous time, and we show that relaxation times in this sense are lower bounded by the inverse of the singular value gap of the generator. {Our concept turns out to be related but not identical to the definition of relaxation times for ergodic averages of discrete Markov chains} given by Chatterjee in a preprint that appeared during the preparation of this paper \cite{Chat2023Spectral}. In \Cref{sec:lowerbounds} we show that relaxation times of second-order lifts in our sense are always lower bounded by the square root of the relaxation times of the underlying reversible diffusion process. This generalises a related result in \cite{Chen1999Lift} to the continuous time case. Similarly to the discrete time case, we introduce {a} concept of optimal lifts which achieve the lower bound on the relaxation time up to a constant factor, and we show that in the Gaussian case, both critical Langevin dynamics and critical randomised Hamiltonian Monte Carlo are optimal lifts in this sense. Finally, we show in \Cref{sec:upperbounds} that the recently developed approach to quantitative hypocoercivity based on space-time Poincaré inequalities \cite{Albritton2021Variational,Cao2019Langevin,Lu2022PDMP} can be rephrased and simplified in the language of lifts and be applied to find optimal lifts in more general scenarios.

\section{Lifts of reversible diffusions}

    In the following, we consider a reversible diffusion process with state space $\S$ and a time-homogenous Markov process with state space $\hat\S=\S\times\V$ and invariant probability measures $\mu(\diff x)$ and $\hat\mu(\diff x\diff v)=\mu(\diff x)\kappa(x,\diff v)$, respectively. We frequently think of the components of the latter as a position and a velocity or momentum variable. The associated Markov semigroups acting on $L^2(\mu)$ and $L^2(\hat\mu)$ are $(P_t)_{t\geq0}$ and $(\hat P_t)_{t\geq 0}$, and their infinitesimal generators are denoted by $(L,\dom(L))$ and $(\hat L, \dom(\hat L))$, respectively. While the underlying diffusion is assumed to be reversible, the second process will not be reversible in general, as an improvement in convergence rate is usually achieved through non-reversibility.
    The Dirichlet form associated to $(L,\dom(L))$ is the extension of
    \begin{equation*}
        \Ecal(f,g) = -\int_\S fLg\diff\mu.
    \end{equation*}
    to a closed symmetric bilinear form with domain $\dom(\Ecal)$ given by the closure of $\dom(L)$ with respect to the norm $\norm{f}_{L^2(\mu)}+\Ecal(f,f)^{1/2}$.

    Motivated by the introductory example, we define second-order lifts of reversible diffusions in the following way.\pagebreak

    \begin{defi}[Second-order lift]\label{def:lifts}\mbox{}
    \begin{enumerate}[(i)]
        \item\label{def:lifti} The semigroup $(\hat P_t)_{t\geq 0}$ is a \emph{second-order lift} of $(P_t)_{t\geq 0}$ if 
        \begin{equation}\label{eq:deflift0}
            \dom(L)\subseteq\{f\in L^2(\mu)\colon f\circ\pi\in\dom(\hat L)\}
        \end{equation}
        and for all $f,g\in\dom(L)$ we have
        \begin{equation}\label{eq:deflift1}
            \int_{\hat\S}\hat L(f\circ \pi)(g\circ\pi)\diff\hat\mu=0
        \end{equation}
        and
        \begin{equation}\label{eq:deflift2}
            \frac{1}{2}\int_{\hat\S}\hat L(f\circ\pi)\hat L(g\circ\pi)\diff\hat\mu =\Ecal(f,g)\,.
        \end{equation}
        \item\label{def:liftii} The semigroup $(\hat P_t)_{t\geq 0}$ is a \emph{strong second-order lift} of $(P_t)_{t\geq 0}$ if 
        \begin{equation}\label{eq:deflift0strong}
            \dom(\Ecal)=\{f\in L^2(\mu)\colon f\circ\pi\in\dom(\hat L)\}
        \end{equation}
        and \eqref{eq:deflift1} and \eqref{eq:deflift2} are satisfied for all $f,g\in\dom(\Ecal)$. 
    \end{enumerate}
    \end{defi}

    We also say that $(\hat L,\dom(\hat L))$ is a (strong) second-order lift of $(L,\dom(L))$. Since we only consider second-order lifts, we may also simply refer to second-order lifts as lifts. In \Cref{thm:lowerbound} below, we will give a lower bound of the relaxation time that holds for arbitrary second-order lifts.

    \begin{rema}
    \begin{enumerate}[(i)]\label{rem:liftsdef}
        \item In general, a non-reversible Markov process can be a lift in the sense of \Cref{def:lifts} \eqref{def:lifti} of several reversible diffusion processes (for example with different boundary behaviours). On the other hand, if $(\hat P_t)_{t\geq 0}$ is a strong lift of $(P_t)_{t\geq0}$, then the Dirichlet form $(\Ecal,\dom(\Ecal))$ and, correspondingly, the semigroup $(P_t)_{t\geq 0}$ are uniquely determined by $(\hat P_t)_{t\geq0}$ through the conditions \eqref{eq:deflift2} and \eqref{eq:deflift0strong}. In this case, the semigroup $(P_t)_{t\geq0}$ is called the \emph{collapse} of $(\hat P_t)_{t\geq 0}$.
        
        \item In \Cref{def:lifts}, it is sufficient to verify conditions \eqref{eq:deflift1} and \eqref{eq:deflift2} for functions $f$ and $g$ from a core of the generator $(L,\dom(L))$ or from a core of the associated Dirichlet form. Similarly, it is sufficient to check that
        \begin{equation*}
            \int_\V\hat L(f\circ\pi)(x,v)\kappa(x,\diff v)=0
        \end{equation*}
        and
        \begin{equation*}
            \frac{1}{2}\int_{\hat\S}\left(\hat L(f\circ\pi)\right)^2\diff\hat\mu=\Ecal(f,f)
        \end{equation*}
        by restricting to the diagonal.

        \item In direct generalisation of the notion of lifts in discrete time, the semigroup $(\hat P_t)_{t\geq0}$ may be called a \emph{first-order lift} of $(P_t)_{t\geq 0}$ if the associated generators satisfy the conditions $\dom(L)\circ\pi\subseteq \dom(\hat L)$ and
        \begin{equation}\label{eq:firstorderlift}
            \int_{\hat\S}(f\circ\pi)\hat L(g\circ\pi)\diff\hat\mu = \int_\S fLg\diff\mu
        \end{equation}
        for all $f,g\in\dom(L)$. This is for instance the case if the underlying discrete-time Markov chain corresponding to a pure Markov jump process is a lift in the sense of \cite{Chen1999Lift} of another.
        
        \item In practice, it is often convenient to split the generator $\hat L$ into a part that acts only on the velocity variable and a part that acts on both variables, usually the deterministic term in the generator. If the part acting only on the velocity leaves $\hat \mu$ invariant, it is sufficient to check the conditions of \Cref{def:lifts} for the latter part of the generator, since $\hat L$ is only applied to functions of the position variable in \Cref{def:lifts}.

        \item Assume that the lift $(\hat P_t)_{t\geq0}$ of $(P_t)_{t\geq0}$ satisfies 
        \begin{equation*}
            \norm{\hat P_t(f\circ\pi)}_{L^2(\hat\mu)}^2 = \norm{f\circ\pi}_{L^2(\hat\mu)}^2 + o(t^2)
        \end{equation*}
        as $t\to0$ and
        \begin{equation*}
            (\hat P_t(f\circ\pi),g\circ\pi)_{L^2(\hat\mu)} = (f\circ\pi,\hat P_t(g\circ\pi))_{L^2(\hat\mu)}
        \end{equation*}
        for all $f,g\in\dom(L)$. The first condition is true if $\hat L$ is antisymmetric and, more generally, if the symmetric part acts only on the velocity, as can be verified by computing the first two time derivatives. The latter is true under a generalised detailed balance condition, i.e.\ that there is a measurable involution $\Theta\colon\hat\S\to\hat\S$ leaving $\hat\mu$ invariant such that  $\pi\circ\Theta = \pi$ and
        \begin{equation*}
            \int_{\hat\S}(f\circ\Theta)\hat L(g\circ\Theta)\diff\hat\mu = \int_{\hat\S}f\hat Lg\diff\hat\mu
        \end{equation*}
        for all $f,g\in\dom(\hat L)$.
        Then 
        \begin{align*}
            \MoveEqLeft\lim_{t\to0}\frac{1}{t^2}\big(f,g-P_{t^2}g)_{L^2(\mu)} = \Ecal(f,g) = \frac{1}{2}\big(\hat L(f\circ\pi),\hat L(g\circ\pi)\big)_{L^2(\hat\mu)}^2\\
            &= \lim_{t\to0}\frac{1}{2t^2}\big(\hat P_t(f\circ\pi)-f\circ\pi,\hat P_t(g\circ\pi)-g\circ\pi\big)_{L^2(\hat\mu)}^2 \\
            &= \lim_{t\to0}\frac{1}{t^2}\big(f\circ\pi,g\circ\pi-\hat P_t(g\circ\pi)\big)_{L^2(\hat\mu)}\,.
        \end{align*}
        In particular, 
        \begin{equation}\label{eq:equivchar}
            \lim_{t\to0}\frac{1}{t^2}\big(f,P_{t^2}g)_{L^2(\mu)} = \lim_{t\to0}\frac{1}{t^2}\big(f\circ\pi,\hat P_t(g\circ\pi)\big)_{L^2(\hat\mu)}
        \end{equation}
        for all $f,g\in\dom(L)$. Hence under the above assumption,
        \begin{equation*}
            \int_\V \hat P_t(g\circ\pi)(x,v)\kappa(x,\diff v) = (P_{t^2}g)(x)+o(t^2)
        \end{equation*}
         in $L^2(\mu)$ as $t\to 0$, is equivalent to the lift conditions \eqref{eq:deflift1} and \eqref{eq:deflift2}. If $(\hat P_t)_{t\geq0}$ and $(P_t)_{t\geq0}$ are the transition semigroups of stationary Markov processes $(\hat X_t,\hat V_t)_{t\geq0}$ and $(X_t)_{t\geq0}$, respectively, then, in view of \eqref{eq:equivchar}, this is also equivalent to
         \begin{equation*}
             \Ex[f(\hat X_0)g(\hat X_t)] = \Ex[f(X_0)g(X_{t^2})] + o(t^2)
         \end{equation*}
         for all $f,g\in\dom(L)$ as $t\to0$.
    \end{enumerate}
        
    \end{rema}
    
    \begin{exam}[Lifts of overdamped Langevin diffusions]\label{ex:ODLifts}
        Our main running example is provided by lifts of overdamped Langevin diffusions. Given a potential $U\in C^1(\R^d)$
        such that $U(x)\to\infty$ as $|x|\to\infty$ and $\int \exp (-U(x))\, \diff x<\infty$, the overdamped Langevin diffusion is the process with state space $\S=\R^d$ given by the SDE
        \begin{equation*}
            \diff Z_t = -\frac{1}{2}\nabla U(Z_t)\diff t+\diff B_t\, ,
        \end{equation*}
        where $(B_t)_{t\geq 0}$ is a $d$-dimensional standard Brownian motion.
        One can show that a unique non-explosive strong solution exists for any initial distribution, its invariant measure is the Boltzmann-Gibbs measure
        \begin{equation*}
            \mu(\diff x)\propto\exp(-U(x))\diff x\, ,
        \end{equation*}
        and the generator on $L^2(\mu)$ is given by
        \begin{equation*}
            Lg = -\frac{1}{2}\nabla U\cdot\nabla g+\frac{1}{2}\Delta g = -\frac{1}{2}\nabla^*\nabla g
        \end{equation*}
        for all $g\in C_0^\infty(\R^d)$, which is a core of the generator \cite{Sturm1994Dirichlet, Stannat1999Dirichlet,Wielens1985Selfadjointness}.
        Here $\nabla^*$ denotes the adjoint of $\nabla$ in $L^2(\mu)$. 
        The following Markov processes are all lifts of overdamped Langevin diffusions:
        \begin{enumerate}[(i)]
        \item \emph{Deterministic Hamiltonian dynamics} associated to the Hamiltonian $H(x,v)=U(x)+\frac{1}{2}|v|^2$ is the Markov process $(X,V)$ with state space $\hat\S = \R^d\times\R^d$ given by the ordinary differential equation
        \begin{align*}
            \diff X_t&=V_t\diff t\,,\\
            \diff V_t&=-\nabla U(X_t)\diff t\,,
        \end{align*}
        together with an initial distribution for $(X_0,V_0)$.
        As Hamiltonian dynamics preserves both the volume
        on $\hat\S$ and the total energy $H(x,v)$, the process is non-explosive and the probability measure  
        \begin{equation*}
            \hat\mu(\diff x\diff v)=\mu (\diff x)\,\kappa (x,\diff v)\propto\exp(-H(x,v))\diff x\diff v\quad\text{ with }\kappa=\mathcal{N}(0,I_d)
        \end{equation*}
        is invariant.
        The generator on $L^2(\hat\mu )$ is
        \begin{equation*}
            \hat L f(x,v)=v\cdot\nabla_xf(x,v)-\nabla U(x)\cdot\nabla_vf(x,v)
        \end{equation*}
        with domain $\dom(\hat L)=\{f\in H^{1,2}_\loc(\R^d\times\R^d)\colon \hat L f\in L^2(\hat\mu)\}$. Clearly we have $C_0^\infty(\R^d)\circ\pi\subseteq\dom(\hat L)$ and for $f\in C_0^\infty(\R^d)$,
        \begin{align*}
            \frac{1}{2}\int_{\R^{2d}}(\hat L(f\circ\pi))^2\diff\hat\mu&=\frac{1}{2}\int_{\R^{2d}}(v\cdot\nabla f)^2\diff\hat\mu=\frac{1}{2}\int_{\R^d}\nabla f^\top\int_{\R^d}vv^\top\kappa(\diff v)\nabla f\diff\mu\\
            &=\frac{1}{2}\int_{\R^d}\nabla f\cdot\nabla f\diff\mu=\frac{1}{2}\int_{\R^d}f\nabla^*\nabla f\diff\mu=-\int_{\R^d}fLf\diff\mu\,.
        \end{align*}
        Hence Hamiltonian dynamics is a second-order lift of the overdamped Langevin diffusion.       
            \item \emph{Langevin dynamics} is the diffusion process $(X,V)$ solving the SDE
            \begin{align*}
                \diff X_t&=V_t\diff t\,,\\
                \diff V_t&=-\nabla U(X_t)\diff t-\gamma V_t\diff t+\sqrt{2\gamma}\diff B_t\,,
            \end{align*}
            where $(B_t)_{t\geq0}$ is a $d$-dimensional Brownian motion and the friction parameter $\gamma$ is a non-negative constant, see e.g.\ \cite{pavliotis2014stochastic}. Its invariant measure is $\hat\mu=\mu\otimes\kappa$. Its generator is given by 
            $\hat L^{(\gamma )} = \hat L+\gamma L_\OU$ where $L_\OU=-v\cdot\nabla_v+\Delta_v$ is the generator of an Ornstein-Uhlenbeck process in the velocity variable which again leaves $\hat\mu$ invariant.  Since $L_\OU(f\circ\pi)=0$ for all $f\in\dom(L)$, Langevin dynamics is a lift of the overdamped Langevin diffusion for any choice of the friction parameter $\gamma\geq 0$.
    
            \item \emph{Randomised Hamiltonian Monte Carlo (RHMC)} runs the Hamiltonian dynamics for an exponentially distributed time $T\sim\mathrm{Exp}(\gamma)$ and then performs a complete velocity refreshment according to $\kappa=\mathcal{N}(0,I_d)$, see \cite{Bou-Rabee2017RHMC}. It leaves invariant the probability measure $\hat\mu=\mu\otimes\kappa$, and its generator on $L^2(\hat\mu)$ is given by $\hat L^{(\gamma )} = \hat L+R$ where $R=\gamma(\Pi_v-I)$ with refresh rate $\gamma\ge 0$ and 
            \begin{equation*}
                (\Pi_vf)(x,v)=\int_{\R^d}f(x,w)\kappa(\diff w)\,.
            \end{equation*}
            The associated equation $\partial_tf=\big(\hat L^{(\gamma )}\big)^*f$, where the adjoint is with respect to $\hat\mu$, modelling the motion of a particle influenced by an external potential $U$ and by random collisions with other particles with Gaussian velocities \cite{Monmarché2020Boltzmann}, is also known as the linear Boltzmann equation in statistical physics, see e.g.\ \cite{Han-Kwan2015Boltzmann,Herau2006Boltzmann}.
            Clearly $R$ leaves $\hat\mu$ invariant. Again we see that $R(f\circ\pi)=0$ for all $f\in\dom(L)$, so that RHMC is a lift of the overdamped Langevin diffusion for any refresh rate $\gamma\geq 0$. 
            
            \item The \emph{Bouncy Particle Sampler (BPS)}, see e.g.\ \cite{Bouchard-Cote2018BPS}, runs in a straight line until its velocity is reflected at a level set of the potential $U$. To ensure ergodicity, one adds velocity refreshments according to $\kappa=\mathcal{N}(0,I_d)$ with rate $\gamma$ as for RHMC. Under additional assumptions, see \cite{Durmus2021PDMP}, $C_0^\infty(\R^d\times\R^d)$ is a core for the resulting generator
            \begin{equation*}
                \hat L^{(\gamma )}f = v\cdot\nabla_xf+(v\cdot\nabla U(x))_+(B-I)f+\gamma(\Pi_v-I)f\,.
            \end{equation*}
            Here $Bf(x,v)=f(x,v-2n(x)n(x)^\top v)$ where $n(x)=\frac{\nabla U(x)}{|\nabla U(x)|}$ if $\nabla U(x)\not=0$ and $n(x)=0$ otherwise. The probability measure $\hat\mu=\mu\otimes\kappa$ is invariant for this process. Since $\hat L^{(\gamma )}(f\circ\pi)=v\cdot (\nabla f)\circ\pi$, one again immediately sees that the BPS is a lift of the overdamped Langevin diffusion. 
        \end{enumerate}
        
         Similarly, other piecewise deterministic Markov processes introduced recently in the MCMC literature are also lifts of an overdamped Langevin diffusion, for example
         Event Chain Monte Carlo \cite{BernardKrauthWilson2009ECMC, MichelKapferKrauth2014ECMC,Krauth2021ECMC} and the Zig-zag process \cite{Bierkens2019ZigZag,Bierkens2017CurieWeiss,Lu2022PDMP}.
    \end{exam}

    \begin{exam}[Lifts of overdamped Riemannian Langevin diffusions]
        We equip $\R^d$ with a Riemannian metric $g\colon\R^d\to\R^{d\times d}$ which for any $x\in\R^d$ defines the inner product
        \begin{equation*}
            \langle v,w\rangle_x = v\cdot g(x)w
        \end{equation*}
        for all $v,w\in\R^d$.
        We write
        \begin{align*}
            \grad f&= g^{-1}\nabla f\,,\\
            \divg F &= \nabla\cdot F+\frac12\sum_{j=1}^d\tr\big(g^{-1}\partial_jg\big)F_j\,, 
        \end{align*}
        for the gradient of a smooth function $f$ and the divergence of a smooth vector field $F$ on the Riemannian manifold $(\R^d,g)$. By Jacobi's formula, these operators satisfy the integration by parts identity
        \begin{equation*}
            \int_{\R^d}\langle \grad h,F\rangle\diff\nu_g = \int_{\R^d}\nabla h\cdot F\diff\nu_g = -\int_{\R^d}h\divg F\diff\nu_g
        \end{equation*}
        for all $h\in C_0^\infty(\R^d)$ and $F\in C_0^\infty(\R^d,\R^d)$, where $\nu_g(\diff x)=\sqrt{\det g(x)}\,\diff x$ is the associated Riemannian volume measure.

        Considering probability measures $\mu(\diff x)\propto\exp(-U(x))\nu_g(\diff x)$ for some continuously differentiable potential $U\colon\R^d\to\R$
        leads to the adjoint
        \begin{equation*}
            \grad^*(F) = -\exp(U)\divg(F\exp(-U))=-\divg F+\langle\grad U,F\rangle
        \end{equation*}
        of $\grad$ with respect to $\mu$. The overdamped Riemannian Langevin diffusion is the solution to the SDE
        \begin{equation}\label{eq:RiemannianOverdamped}
            \diff X_t=\frac{1}{2}\big({-(\grad U)(X_t)}+(\divg g^{-1})(X_t)\big)\diff t+\sqrt{g^{-1}(X_t)}\diff B_t\,,
        \end{equation}
        where $(B_t)_{t\geq0}$ is a $d$-dimensional standard Brownian motion and the divergence is applied row-wise to $g^{-1}$. 
        Under appropriate assumptions on $U$ and $g$, the process is non-explosive and reversible with respect to its invariant measure $\mu$, see \cite{Vempala2022Riemannian,Bakry2014Analysis}. On smooth compactly supported functions, the generator of the corresponding transition semigroup on $L^2(\mu)$ is given by
        \begin{align*}
            Lf&=\frac{1}{2}\big\langle{-\grad U}+\divg g^{-1},\grad f\big\rangle+\frac{1}{2}g^{-1}\mathbin{:}\nabla^2f\\
            &=\frac{1}{2}\divg\grad f-\frac{1}{2}\langle\grad U,\grad f\rangle
            = -\frac{1}{2}\grad^*(\grad f)\,.
        \end{align*}

        Interpreting $g$ as a position-dependent mass matrix leads to a Hamiltonian dynamics associated to the Hamiltonian 
        \begin{equation*}
            H(x,p)=U(x)+\frac12p^\top g^{-1}(x)p
        \end{equation*}
        on $\R^d\times\R^d$, where $(x,p)$ are position and momentum. It reads
        \begin{equation}\label{eq:RiemannianHamiltonian}
            \begin{pmatrix}\dot x\\\dot p\end{pmatrix}=\begin{pmatrix}0&I_d\\-I_d&0\end{pmatrix}\nabla H(x,p) = \begin{pmatrix}g^{-1}(x)p\\-\nabla U(x)-\frac12\nabla_x\left(p^\top g^{-1}(x)p\right)\end{pmatrix}.
        \end{equation}
        By construction it leaves the measure $\hat\mu(\diff x\diff p)=\mu(\diff x)\mathcal{N}\big(0,g(x)\big)(\diff p)$ invariant. For $f\in C_0^\infty (\R^d\times\R^d)$, its generator on $L^2(\hat\mu)$ is given by
        $$
            \hat Lf(x,p)=p^\top g^{-1}(x)\nabla_xf(x,p)-\nabla U(x)^\top\nabla_pf(x,p)-\frac12\left(\nabla_x\left(p^\top g^{-1}(x)p\right)\right)^\top\nabla_p f(x,p)\,.
        $$
        Denoting $\kappa(x,\cdot)=\mathcal{N}\big(0,g(x)\big)$ we have $\hat L(f\circ\pi)(x,p)=p^\top\nabla_gf(x)$ and thus
        \begin{equation*}
            \int_{\R^d}\hat L(f\circ\pi)(x,p)\kappa(x,\diff p)=0\,,
        \end{equation*}
        and
        \begin{align*}
            \frac{1}{2}\int_{\R^d\times\R^d}(\hat L(f\circ\pi))^2\diff\hat\mu&=\frac{1}{2}\int_{\R^d}\int_{\R^d}\grad f(x)^\top p p^\top\grad f(x)\,\kappa(x,\diff p)\mu(\diff x)\\
            &=\frac{1}{2}\int_{\R^d}\langle\grad f,\grad g\rangle\diff\mu=-\int_{\R^d}fLf\diff\mu
        \end{align*}
        for all $f\in C_0^\infty(\R^d)$ and $x\in\R^d$. Hence the Hamiltonian dynamics \eqref{eq:RiemannianHamiltonian} is a lift of the overdamped Riemannian Langevin diffusion \eqref{eq:RiemannianOverdamped}.

        As in \Cref{ex:ODLifts}, the Hamiltonian dynamics can be modified by passing to the generator $\hat L+\gamma R$, where $\gamma\ge 0$ and $R$ acts only the momentum and leaves $\kappa(x,\diff p)$ invariant. The choice $R=-\nabla_p^*\nabla_p=\Delta_p-g^{-1}\nabla_p$, where the adjoint is with respect to $\hat\mu$, leads to Riemannian Langevin dynamics given by the solution $(X_t,P_t)$ to the SDE
        \begin{align*}
            \diff X_t&=g^{-1}(X_t)P_t\diff t\,,\\
            \diff P_t&=\left(-\nabla U(X_t)-\frac12\nabla_x\left(P_t^\top g^{-1}(X_t)P_t\right)\right)\diff t-\gamma g^{-1}(X_t)P_t\diff t+\sqrt{2\gamma}\diff B_t\,.
        \end{align*}
        By the same argument as above, this is again a lift of the underlying overdamped Riemannian Langevin diffusion.
    \end{exam}    

\section{Relaxation times of non-reversible Markov processes}\label{sec:relaxation}

    Let $L_0^2(\mu)=\{f\in L^2(\mu)\colon\int f\diff\mu=0\}$ be the orthogonal complement of the constant functions.
    The transition semigroup $(P_tf)(x)=\int f(y)p_t(x,\diff y)$ of a \emph{reversible} Markov process with invariant measure $\mu$ satisfies
    \begin{equation}\label{eq:contr}
        \norm{P_t}_{L_0^2(\mu)\to L_0^2(\mu)} = \sup_{f\in L_0^2(\mu)\setminus\{0\}}\frac{\norm{P_tf}_{L^2(\mu)}}{\norm{f}_{L^2(\mu)}} = e^{-\lambda t}\,,
    \end{equation}
    with the decay rate $\lambda$ given by the spectral gap
    \begin{equation}\label{eq:gap}
        \gap(L)=\inf\{\Re(\alpha)\colon\alpha\in\spec(-L|_{L_0^2(\mu)})\}
    \end{equation}
    of the generator $(L,\dom(L))$ of $(P_t)_{t\geq 0}$ in the Hilbert space $L^2(\mu)$.
    Equivalently, the $\chi^2$-divergence $\chi^2(\nu p_t|\mu)$ between the law of the process at time $t$ and the invariant measure decays exponentially with rate $\lambda$. Therefore, in this reversible case, decay by a factor of $\eps>0$ has taken place after time $\frac{1}{\lambda}\log(\eps^{-1})$, so that it is natural to define the $L^2$-relaxation time $t_\rel$ as the inverse of the spectral gap \cite{LPW2017markov,Saloff-Coste1997Lectures}.
    
    In contrast, for \emph{non-reversible} Markov processes, the operator norm $\norm{P_t}_{L_0^2(\mu)\to L_0^2(\mu)}$ is no longer a pure exponential, and one can only expect bounds of the form 
    \begin{equation}\label{eq:decay}
        \norm{P_t}_{L_0^2(\mu)\to L_0^2(\mu)}\leq Ce^{-\nu t}
    \end{equation}
    with constants $C\in[1,\infty)$ and $\nu\leq\gap(L)$, see e.g.\ \cite{engel1999semigroups}. This motivates the following definition of an $\eps$-relaxation time.
    \begin{defi}[Non-asymptotic relaxation time]\label{def:relaxationtime}
        Let $\eps\in(0,1)$. The \emph{$\eps$-relaxation time} $t_\rel(\eps)$ of $(P_t)_{t\geq 0}$ is
        \begin{equation*}
            t_\rel(\eps) = \inf\{t\geq0\colon\norm{P_tf}_{L^2(\mu)}\leq \eps\norm{f}_{L^2(\mu)}\textup{ for all }f\in L_0^2(\mu)\}
        \end{equation*}
        and the \emph{relaxation time} is $t_\rel = t_\rel(e^{-1})$.
    \end{defi}
    \begin{rema}[Asymptotic decay rate and relaxation times]\mbox{}
        \begin{enumerate}[(i)]
            \item The $\eps$-relaxation time as a function $\eps\mapsto t_\rel(\eps)$ is the inverse function of $t\mapsto\norm{P_t}_{L_0^2(\mu)\to L_0^2(\mu)}$. Therefore, the asymptotic decay rate
            \begin{equation*}
                \omega = -\lim_{t\to\infty}\frac{1}{t}\log\norm{P_t}_{L_0^2(\mu)\to L_0^2(\mu)}
            \end{equation*}
            of the semigroup restricted to $L_0^2(\mu)$ can be expressed as
            \begin{equation}\label{decayrate}
                \frac{1}{\omega} = \lim_{\eps\to 0}\frac{t_\rel(\eps)}{\log(\eps^{-1})}\,.
            \end{equation}
            Hence $\omega^{-1}$ can also be interpreted as an \emph{asymptotic relaxation time}. The relation $\gap(L)=\omega$ only holds true under additional assumptions ensuring a spectral mapping theorem relating the spectrum of $P_t$ to that of $L$, see \cite{engel1999semigroups}. In general, only $\gap(L)\geq\omega$ is true. More importantly, the spectral gap cannot be used to bound the constant $C$ in \eqref{eq:decay} and can hence only provide asymptotic estimates, whereas $t_\rel(\eps)$ is a non-asymptotic quantity.
            \item  For reversible processes, \eqref{eq:contr} shows $t_\rel(\eps) = \log(\eps^{-1})/{ \gap(L)}$ and $t_\rel$ hence coincides with the usual definition of the relaxation time as the inverse of the spectral gap. 
            For non-reversible processes, the semigroup property only yields the upper bounds
            $$t_\rel(\delta )\leq t_\rel(\eps )\cdot \frac{\log(\delta^{-1})}{\log(\eps^{-1})}\quad\text{ for }\delta\le\eps \, .$$
            In particular, the limit in \eqref{decayrate} is monotone, and thus
            \begin{equation*}
                \frac{1}{\omega} \le \frac{t_\rel(\eps)}{\log(\eps^{-1})}\quad\text{ for any } \eps \in (0,1)\, .
            \end{equation*}
        \end{enumerate} 
    \end{rema}

    To obtain quantitative upper bounds on the relaxation time, it will be convenient to reformulate \eqref{eq:decay} in the following way.
    \begin{defi}[Delayed exponential contractivity]\label{def:delayedcontr}
        Let $T,\nu\in(0,\infty)$ and let $(P_t)_{t\geq 0}$ be a Markov semigroup with invariant probability measure $\mu$.
        \begin{enumerate}[(i)]
            \item The semigroup $(P_t)_{t\geq 0}$ is called \emph{$T$-delayed exponentially contractive with rate $\nu$} if
            \begin{equation}\label{eq:delayedcontr}
                \norm{P_tf}_{L^2(\mu)}\leq e^{-\nu(t-T)}\norm{f}_{L^2(\mu)}
            \end{equation}
            for all $f\in L^2_0(\mu)$ and $t\geq 0$.
            \item The semigroup $(P_t)_{t\geq 0}$ is called \emph{exponentially contractive in $T$-average with rate $\nu$} if
            \begin{equation}\label{eq:averagecontr}
                \frac{1}{T}\int_t^{t+T}\norm{P_sf}_{L^2(\mu)}\diff s\leq e^{-\nu t}\norm{f}_{L^2(\mu)}
            \end{equation}
            for all $f\in L^2_0(\mu)$ and $t\geq 0$.
        \end{enumerate}
    \end{defi}
    
        If $(P_t)_{t\geq 0}$ is exponentially contractive in $T$-average with rate $\nu$, then for $t\geq T$,
            \begin{equation*}
                \norm{P_tf}_{L^2(\mu)}\leq\int_{t-T}^t\norm{P_sf}_{L^2(\mu)}\diff s\leq e^{-\nu(t-T)}\norm{f}_{L^2(\mu)}\,.
            \end{equation*}
            For $t<T$ one has $e^{-\nu(t-T)}>1$ so that exponential contractivity in $T$-average implies $T$-delayed exponential contractivity with the same rate.
            For $\eps\in (0,1)$ let
            \begin{equation*}
                t_0(\eps)\ :=\ \inf\left\{\frac{1}{\nu}\log(\eps^{-1})+T\colon\nu,T\in[0,\infty)\textup{ such that \eqref{eq:delayedcontr} holds} \right\}\,.
                \end{equation*}
            We can compare the $\eps$-relaxation time to the time
            $t_0(\eps )$.

            \begin{lemm}\label{lem:epsrelaxation}For any $\eps\in (0,1)$,
            $$t_\rel(\eps)\leq t_0(\eps )\le 2t_\rel(\eps)\, .$$
            \end{lemm}
            \begin{proof}
                On the one hand, if \eqref{eq:delayedcontr} holds, then the $L^2$-norm has decayed by a factor $\eps$ after time $T+\frac{1}{\nu}\log(\eps^{-1})$. This yields the bound
                $t_\rel(\eps)\leq t_0(\eps )$.
                On the other hand, for any $t\geq 0$ and $f\in L_0^2(\mu)$,
                \begin{align*}
                    \norm{P_tf}_{L^2(\mu)}&\leq\norm{P_{t_\rel(\eps)}^{\lfloor t/t_\rel(\eps)\rfloor}f}_{L^2(\mu)}\leq \eps^{\lfloor t/t_\rel(\eps)\rfloor}\norm{f}_{L^2(\mu)}\\
                    &\leq e^{-\log(\eps^{-1})t_\rel(\eps)^{-1}(t-t_\rel(\eps))}\norm{f}_{L^2(\mu)}\,,
                \end{align*}
                so that \eqref{eq:delayedcontr} holds with $\nu = \log(\eps^{-1})t_\rel(\eps)$ and $T = t_\rel(\eps)$. This shows that $t_0(\eps )\le 2t_\rel(\eps) $.
            \end{proof}

    In the non-reversible case, the spectral gap does not provide a lower bound on the non-asymptotic relaxation times introduced above.
    However, we will show that lower bounds on $t_\rel$ are given by the inverse of the \emph{singular value gap}
    \begin{equation}\label{singularvaluegap}
        \sing(L)=\inf\left\{\frac{\norm{Lf}_{L^2(\mu)}}{\norm{f}_{L^2(\mu)}}\colon f\in L_0^2(\mu)\cap\dom(L), f\not=0\right\}\,.
    \end{equation}
    The singular value gap is the infimum of the singular spectrum of the generator $L$ restricted to the complement of the constant functions, i.e.\ it is the square root of the spectral gap of $L^*L$. In the reversible case, $\sing(L)=\gap(L)$. In the non-reversible case, one only knows that $|\lambda|\geq\sing(L)$ for any $\lambda\in\spec(L)\setminus\{0\}$, but this does not imply a relation between $\sing(L)$ and $\gap(L)$.

    If the operator $-L\colon L_0^2(\mu)\cap\dom(L)\to L_0^2(\mu )$ is invertible then its inverse $G$ is the potential operator and $\sing(L)$
    is the inverse of the operator norm of $G$, i.e.
    \begin{equation}
        \frac 1{\sing(L)}=\sup\left\{\frac{\norm{Gf}_{L^2(\mu)}}{\norm{f}_{L^2(\mu)}}\colon f\in L_0^2(\mu), f\not=0\right\}\,.
    \end{equation}
    In this case, $Gf=\int_0^\infty P_tf\diff t$ for any $f\in L_0^2(\mu )$, and in stationarity, the variance of ergodic averages $\frac{1}{t}\int_0^t f(X_s)\diff s$ is bounded from
    above by $\frac{2}{t\sing(L)}\norm{f}^2_{L^2(\mu)}$, see e.g.\ \cite{komorowski2012fluctuations} or \cite{EberleMP}.
    
    \begin{rema}[Chatterjee's concept of relaxation times]\label{rem:singgap}
        In discrete time and on a finite state space, Chatterjee \cite{Chat2023Spectral} recently defined the relaxation time for the ergodic averages of a non-reversible Markov chain as the inverse of the singular value gap.
        Our definition of the relaxation time $t_\rel$ is different and more closely related to mixing times. The inverse of the singular value gap only provides a lower bound for $t_\rel$ in the non-reversible case.
    \end{rema}

    \begin{lemm}\label{lem:lowerbounds}
        Assume $(P_t)_{t\geq 0}$ is $T$-delayed exponentially contractive with rate $\nu$ or exponentially contractive in $2T$-average with rate $\nu$. Then
        \begin{equation*}
            \frac{1}{\nu}\geq\gap(L)^{-1}\qquad\textup{and}\qquad\frac{1}{\nu}+T\geq\sing(L)^{-1}\,.
        \end{equation*}
        In particular, $t_\rel\geq\frac{1}{2}\sing(L)^{-1}$.
    \end{lemm}
    \begin{proof}
        Let $\alpha\in\C$ with $\Re(\alpha)<\nu$. For $f\in L_0^2(\mu)$, both the $T$-delayed exponential contractivity and exponential contractivity in $2T$-average with rate $\nu$ imply
        \begin{equation*}
            \int_0^\infty |e^{\alpha t}|\norm{P_tf}_{L^2(\mu)}\diff t\leq C(\alpha)\norm{f}_{L^2(\mu)}\,,
        \end{equation*}
        where $C(\alpha)$ is a finite constant depending only on $\alpha$. Therefore, the resolvent $(-\alpha I-L)^{-1}f=\int_0^\infty e^{\alpha t}P_tf\diff t$ exists in $L^2(\mu)$ and defines a bounded linear operator on $L_0^2(\mu)$, i.e.\ $\alpha\not\in\spec(-L|_{L_0^2(\mu)})$ and hence $\gap(L)\geq\nu$, showing the first assertion.
        
        For $f\in L^2_0(\mu)$ the exponential contractivity in $2T$-average with rate $\nu$ yields
        \begin{align*}
            \int_0^\infty\norm{P_tf}_{L^2(\mu)}\diff t&=\int_0^\infty\frac{1}{2T}\int_t^{t+2T}\norm{P_sf}_{L^2(\mu)}\diff s\diff t+\frac{1}{2T}\int_0^{2T}(2T-t)\norm{P_tf}_{L^2(\mu)}\diff t\\
            &\leq\left(\int_0^\infty e^{-\nu t}\diff t+\frac{1}{2T}\int_0^{2T}(2T-t)\diff t\right)\norm{f}_{L^2(\mu)}\\
            &=\left(\frac{1}{\nu}+T\right)\norm{f}_{L^2(\mu)}\,.
        \end{align*}
        This yields convergence of the potential operator $Gf=\int_0^\infty P_tf\diff t$ in $L^2(\mu)$ for all $f\in L^2_0(\mu)$, which is the inverse of the operator $-L\colon L_0^2(\mu)\cap\dom(L)\to L_0^2(\mu)$, see \cite{Sato1972Potential}, together with the estimate
        \begin{equation*}
            \norm{Gf}_{L^2(\mu)}\leq\left(\frac{1}{\nu}+T\right)\norm{f}_{L^2(\mu)}\,.
        \end{equation*}
        For $g\in\dom(L)\cap L_0^2(\mu)$, we may set $f=Lg$ and obtain the bound
        \begin{equation*}
            \norm{g}_{L^2(\mu)}\leq\left(\frac{1}{\nu}+T\right)\norm{Lg}_{L^2(\mu)}
        \end{equation*}
        for all $g\in\dom(L)\cap L_0^2(\mu)$, which yields the claim.

        In the case of $T$-delayed exponential contractivity with rate $\nu$, we have that
        \begin{align*}
            \int_0^\infty\norm{P_tf}_{L^2(\mu)}\diff t&\leq\int_{T}^\infty e^{-\nu(t-T)}\norm{f}_{L^2(\mu)}\diff t+\int_0^{T}\norm{f}_{L^2(\mu)}\diff t\\
            &=\left(\frac{1}{\nu}+T\right)\norm{f}_{L^2(\mu)}
        \end{align*}
        for all $f\in L^2_0(\mu)$, and we conclude analogously.
    \end{proof}
    In contrast to lower bounds, upper bounds on $t_\rel$ of the correct order are much harder to obtain in the degenerate non-reversible case and require quantitative hypocoercivity results \cite{Albritton2021Variational,Lu2022PDMP,Cao2019Langevin,BrigatiStoltz2023Decay,Brigati2022FokkerPlanck}, see \Cref{sec:upperbounds} below.

\section{Lower bounds for relaxation times of lifts}\label{sec:lowerbounds}

    Suppose that $(P_t)_{t\geq 0}$ with generator $(L,\dom(L))$ is the transition semigroup on $L^2(\S,\mu)$ of a reversible Markov process and that $(\hat P_t)_{t\geq 0}$ with generator $(\hat L,\dom(\hat L))$ is the transition semigroup on $L^2(\S\times\V,\hat\mu)$ of an arbitrary second-order lift. The next \namecref{thm:lowerbound} shows that relaxation to equilibrium, measured by the relaxation times $t_\rel(P)$ and $t_\rel(\hat P)$ of the reversible Markov process and its lift, can at most be accelerated by a square root through lifting.
    \begin{theo}\label{thm:lowerbound}
        If $(\hat P_t)_{t\geq 0}$ is $T$-delayed exponentially contractive with rate $\nu$ or exponentially contractive in $2T$-average with rate $\nu$, then
        \begin{equation*}
            \frac{1}{\nu}+T\geq\frac{1}{\sqrt{2\gap(L)}}\,.
        \end{equation*}
        In particular, $$t_\rel(\hat P)\geq \frac{1}{2\sqrt{2}}\sqrt{t_\rel(P)}\,.$$
    \end{theo}
    \begin{proof}
        Let $\lambda>\gap(L)$. By the variational characterisation of the spectral gap, there exists $f\in\dom(L)\cap L_0^2(\mu)$ such that $(f,-Lf)_{L^2(\mu)}\leq \lambda\norm{f}_{L^2(\mu)}^2$. By \Cref{def:lifts}, $f\circ\pi\in\dom(\hat L)\cap L_0^2(\hat\mu)$ and
        \begin{equation}\label{eq:singupper}
            \norm{\hat L(f\circ\pi)}_{L^2(\hat\mu)}^2=2(f,-Lf)_{L^2(\mu)}\leq 2\lambda\norm{f}_{L^2(\mu)}^2=2\lambda\norm{f\circ\pi}_{L^2(\hat\mu)}^2\,.
        \end{equation}
        By considering the limit $\lambda\to\gap(L)$, we see that $\sing(\hat L)\leq\sqrt{2\gap(L)}$,
        which implies the assertion by \Cref{lem:lowerbounds}.
    \end{proof}

    \begin{rema}
        \begin{enumerate}[(i)]
            \item Consider a Markov chain in discrete or continuous time with state space $\hat\S$, transition kernel $\hat p$ and invariant measure $\hat\mu$ that is a first-order lift (in the sense of \cite{Chen1999Lift}, or equivalently, \eqref{eq:firstorderlift}) of a reversible Markov chain with state space $\S$, transition kernel $p$ and invariant measure $\mu$. That is,
            \begin{equation*}
                (f\circ\pi,\hat p(g\circ\pi))_{L^2(\hat\mu)} = (f,pg)_{L^2(\mu)}
            \end{equation*}
            for all bounded measurable functions $f,g\colon\S\to\R$.
            Then the associated generators $\hat L = \hat p-I$ and $L=p-I$ satisfy
            \begin{align*}
                \norm{\hat L(f\circ\pi)}_{L^2(\hat\mu)}^2 &= \norm{(\hat p-I)(f\circ\pi)}_{L^2(\hat\mu)}^2\\
                &= \norm{f\circ\pi}_{L^2(\hat\mu)}^2-2(f\circ\pi,\hat p(f\circ\pi))_{L^2(\hat\mu)}+\norm{\hat p(f\circ\pi)}_{L^2(\hat\mu)}^2\\
                &\leq 2\norm{f}_{L^2(\mu)}^2-2(f,pf)_{L^2(\mu)} = 2(f,-Lf)_{L^2(\mu)}\,,
            \end{align*}
            since $\hat p$ is a contraction. Hence the second-order lift condition \eqref{eq:deflift2} holds with an inequality, so that we can conclude $\sing(\hat L)\leq\sqrt{2\gap(L)}$ similarly as in \eqref{eq:singupper}. \Cref{lem:lowerbounds} can then be used to obtain a lower bound on the relaxation time of the lifted Markov chain in terms of the spectral gap of $L$ as in \Cref{thm:lowerbound}.
            
            \item It is an open problem whether also $\frac{1}{\nu}$ can be lower bounded in terms of $1/\sqrt{\gap(L)}$. This would follow from an upper bound on $\gap(\hat L)$ by a multiple of $\sqrt{\gap(L)}$. We have not managed to prove such a bound. This is similar to the discrete time case, where a lower bound for the mixing time is available but an upper bound for the spectral gap of lifts is not known \cite{Chen1999Lift}.
        \end{enumerate}
    \end{rema}
    
    An acceleration of the relaxation time from ${1}/{\gap(L)}$ to a multiple of ${1}/{\sqrt{\gap(L)}}$ corresponds to a diffusive to ballistic speed-up.
    \Cref{thm:lowerbound} motivates the following definition.

    \begin{defi}[$C$-optimal lift]\label{def:optimallift}
        For a fixed constant $C\in[1,\infty)$, a lift as above is called \emph{$C$-optimal}, if  $$t_\rel(\hat P)\leq C\frac{1}{2\sqrt{2}}\sqrt{t_\rel(P)}\,.$$
    \end{defi}

    \begin{exam}[Gaussian case]
        Suppose that $\mu=\mathcal{N}(0,1)$ is the standard normal distribution on $\R$. The overdamped Langevin process solving the SDE 
        \begin{equation*}
            \diff Z_t=-\frac{1}{2}Z_t\diff t+\diff B_t
        \end{equation*}
        driven by a one-dimensional Brownian motion $(B_t)_{t\geq 0}$ is reversible, has invariant measure $\mu$ and spectral gap ${1}/{2}$. The critical Langevin dynamics solving
        \begin{align*}
            \diff X_t&=V_t\diff t\,,\\
            \diff V_t&=-X_t\diff t-\gamma V_t\diff t+\sqrt{2\gamma}\diff B_t\quad\text{ with }\gamma =2
        \end{align*}
        is a second-order lift of this process with invariant measure $\hat\mu=\mu\otimes \mathcal{N}(0,1)$. 
        The choice $\gamma =2 $ maximises the spectral gap
        of the generator, see \Cref{fig:gap}.
        \begin{figure}[h]
            \centering
            \begin{center}
                \includegraphics[height=4cm]{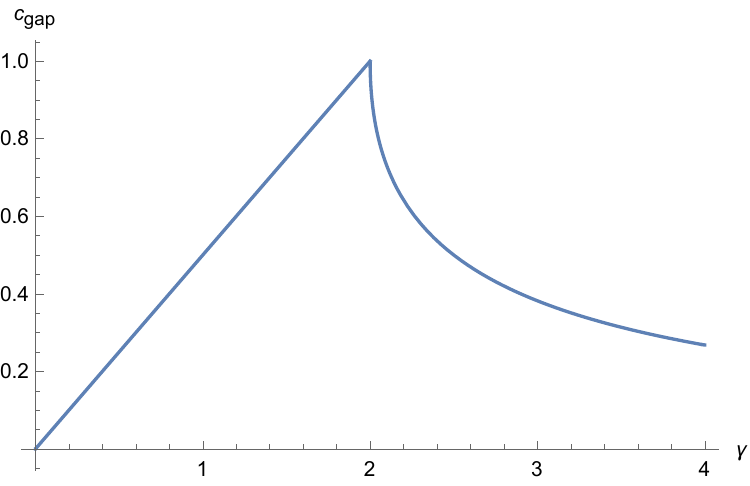} 
            \end{center} 
            \caption{Spectral gap for Langevin dynamics.}\label{fig:gap}
        \end{figure}
        
        Denoting $A=\begin{pmatrix}0&-1\\1&2\end{pmatrix}$, an explicit computation \cite{Arnold2022Propagator} gives
        \begin{equation*}
            \norm{\hat P_t}_{L_0^2(\hat\mu)\to L_0^2(\hat\mu)}=\norm{\exp\left(-tA\right)}=\sqrt{1+2t^2+2t({1+t^2})^{1/2}}\cdot e^{-t}\, 
        \end{equation*}
        for any $t\geq0$, which yields $t_\rel(\hat P) \le 2.73$. Thus in this case, critical Langevin dynamics
        is a $5.46$-optimal lift of the overdamped Langevin diffusion. Note that the asymptotic decay rate of the critical Langevin dynamics is exactly $1$, the square root of twice the spectral gap of the corresponding reversible diffusion.
        
        More generally, for any $m>0$, the critical Langevin dynamics for the potential $U_m(x)={m}x^2/2$ with invariant measure $\hat\mu_m = \mathcal{N}(0,m^{-1})\otimes \mathcal{N}(0,1)$ is a second-order lift of the corresponding overdamped Langevin process with spectral gap ${m}/{2}$. Since the transition semigroup $(\hat P_t^{(m)})_{t\geq 0}$ satisfies
        \begin{equation*}
            \norm{\hat P_t^{(m)}}_{L_0^2(\hat\mu_m)\to L_0^2(\hat\mu_m)} = \norm{\hat P_{\sqrt{m}t}}_{L_0^2(\mu)\to L_0^2(\mu)}\,,
        \end{equation*}
        $t_\rel(\hat P^{(m)}) = m^{-1/2}t_\rel(\hat P)$. This shows that also in this case, critical Langevin dynamics is a $5.46$-optimal lift of the corresponding overdamped Langevin process. 

        Finally, a factorisation argument shows that the same results hold true for arbitrary quadratic potentials on $\mathbb R^d$, i.e., in the Gaussian case, critical Langevin dynamics is \emph{always} a $5.46$-optimal lift of the corresponding overdamped Langevin diffusion.
    \end{exam}

\section{Upper bounds for relaxation times of lifts}\label{sec:upperbounds}
    Based on the approach of Albritton, Armstrong, Mourrat and Novack \cite{Albritton2021Variational} relying on space-time Poincaré inequalities, Cao, Lu and Wang \cite{Cao2019Langevin,Lu2022PDMP} have derived quantitative upper bounds for relaxation times of various non-reversible Markov processes. The framework of second-order lifts developed above allows us to rephrase and simplify their results.
    
    For any $\gamma\geq0$, we consider a Markov process with state space $\R^d\times\R^d$ and invariant measure $\hat\mu=\muk$ where $\kappa =\mathcal N(0,I_d)$. The associated transition semigroups acting on $L^2(\hat\mu)$ are denoted by $(\hat P_t^{(\gamma)})_{t\geq 0}$ and their generators by $(\hat L^{(\gamma)},\dom(\hat L^{(\gamma)}))$. We assume the following decomposition of the generators.
    
    \begin{assu}\label{ass:hatLgamma}
        There is an operator $(\hat L,\dom(\hat L))$ on $L^2(\hat\mu)$ such that the following holds:
        \begin{enumerate}[(i)]
            \item For all $\gamma\geq0$ we have $\dom(\hat L^{(\gamma)})\subseteq\dom(\hat L)$ and
            \begin{equation*}
                \hat L^{(\gamma)} f = \hat Lf + \gamma(\Pi_v-I)f\quad\textup{ for all }f\in\dom(\hat L^{(\gamma)})\,.
            \end{equation*}
            \item The operator $(\hat L,\dom(\hat L))$ is negative definite, i.e.\ $(f,\hat Lf)_{L^2(\hat\mu)}\leq0$ for all $f\in\dom(\hat L)$.
        \end{enumerate}
    \end{assu}

    \Cref{ass:hatLgamma} is for instance satisfied by the generators of Randomised Hamiltonian Monte Carlo or the Bouncy Particle Sampler with refresh rate $\gamma$ by choosing 
    \begin{equation}\label{LhatHamflow}
        \hat Lf = v\cdot\nabla_xf-\nabla U(x)\cdot\nabla_vf
    \end{equation}
    to be the generator of the deterministic Hamiltonian flow, or
    \begin{equation}\label{LhatBPS}
        \hat Lf = v\cdot\nabla_xf+(v\cdot\nabla U(x))_+(B-I)f
    \end{equation}
    to be the generator of the Bouncy Particle Sampler without refreshment, respectively, see \Cref{ex:ODLifts}. In the former case, $\hat L$ is antisymmetric, and in the latter 
    \begin{equation*}
        (f,\hat Lf)_{L^2(\hat\mu)} = -\frac{1}{4}\int_{\R^d}|v\cdot\nabla U(x)|(Rf-f)^2\diff\hat\mu\, ,
    \end{equation*}
    so that in both cases, $\hat L$ is negative definite.

        \subsection{Space-time Poincaré inequalities and contractivity}\label{ssec:STPI}

    Exponential contractivity in $T$-average, and thereby upper bounds on the relaxation times of the transition semigroups $(\hat P_t^{(\gamma)})_{t\geq0}$, can be established using a Poincaré-type inequality in space and time, as introduced by \cite{Albritton2021Variational}. To this end, we fix a finite time $T>0$. It will then be convenient to consider space and time together, and we hence set $$\muq=\lambda|_{[0,T]}\otimes\mu$$ 
    where $\lambda|_{[0,T]}$ is the Lebesgue measure on $[0,T]$. 
    Furthermore, we consider the operator $(A,\dom(A))$ on $L^2([0,T]\times\R^d\times\R^d,\muqk)$ defined by
    \begin{equation}
        \label{def:A} Af=-\partial_tf+\hat Lf
    \end{equation}
    with domain consisting of all functions  
       $ f\in L^2(\muqk)$ such that $f(\cdot,x,v)$ is absolutely continuous for $\muk$-a.e.\ $(x,v)\in\R^d\times\R^d$ with $\partial_t f\in L^2(\muqk)$, and $f(t,\cdot)\in\dom(\hat L)$ for $\lambda\textup{-a.e.\ }t\in[0,T]$ with $\hat Lf\in L^2(\muqk)$.
    
    \begin{defi}[Space-time Poincaré inequality]\label{def:stpi}
        A \emph{space-time Poincaré inequality with constants $C_0(T)$ and $C_1(T)$} holds for $\hat L$, if for all functions $f\in\dom(A)$ 
        \begin{equation*}
            \norm{f-\int f\diff(\muqk)}_{L^2(\muqk)}^2\leq C_0(T)\norm{Af}_{L^2(\muqk)}^2+C_1(T)\norm{f-\Pi_vf}_{L^2(\muqk)}^2\,.
        \end{equation*}
        We also say that $\hat L$ satisfies $\stpi(T,C_0(T),C_1(T))$.
    \end{defi}
    \begin{rema}\label{rem:core}
    By continuity, it suffices to verify the inequality
    in \Cref{def:stpi} for functions $f$ from a \emph{core} of $A$, i.e.\ a subset that is dense in the domain of $A$ with respect to the graph norm
    $\norm{f}_{L^2(\muqk)} + \norm{Af}_{L^2(\muqk)}$.
    For example, if $\mu (\diff x)\propto\exp (-U(x))\diff x$ for
    some function $U\in C^1(\R^d)$ such that $U(x)\to\infty$ as $|x|\to\infty$, and
    $\hat L$ is the generator of the Hamiltonian flow
    given by \eqref{LhatHamflow}, then the space $C_0^\infty([0,T]\times\R^d\times\R^d)$ of smooth compactly supported functions on $[0,T]\times\R^d\times\R^d$ is a core.
    Indeed, a function $f\in\dom(A)$ can be approximated in the graph norm by the compactly supported functions $f_n(t,x,v)=g_n(H(x,v))f(t,x,v) $ where $g_n:[0,\infty )\to [0,1]$ is an increasing sequence of smooth compactly supported functions such that $g_n(s)\to 1$ for all $s\in [0,\infty )$, because $\hat LH=0$ implies $(Af_n)(t,x,v)=g_n(H(x,v))(Af)(t,x,v)$.
    It thus follows by a standard mollification that $C_0^\infty([0,T]\times\R^d\times\R^d)$ is a core for $A$.
    \end{rema}
    
    As observed in \cite{Albritton2021Variational,Cao2019Langevin,Lu2022PDMP},
    a space-time Poincaré inequality for $\hat L$ yields the following quantitative hypocoercivity estimate for the semigroups $(\hat P_t^{(\gamma)})_{t\geq0}$.
    \begin{theo}[STPI implies exponential contractivity in $T$-average]\label{thm:decay}
        Assume that $\hat L$ satisfies $\stpi(T,C_0(T),C_1(T))$. Then for any $\gamma>0$, the semigroup $(\hat P_t^{(\gamma)})_{t\geq0}$ is exponentially contractive in $T$-average with rate
        \begin{equation*}
            \nu = \frac{\gamma}{\gamma^2C_0(T)+C_1(T)}\,.
        \end{equation*}
        The choice $\gamma=\sqrt{\frac{C_1(T)}{C_0(T)}}$ maximises the rate and gives $\nu=\frac{1}{2\sqrt{C_0(T)C_1(T)}}$. For this choice,
        \begin{equation*}
            t_\rel(\eps)\leq T+2\sqrt{C_0(T)C_1(T)}\log(\eps^{-1})\,.
        \end{equation*}
    \end{theo}
    For the reader's convenience we recall the short proof of the theorem.
    \begin{proof}
        Let $f\in L_0^2(\mu)$. Then by Assumption \ref{ass:hatLgamma},
        \begin{align*}
            \MoveEqLeft\frac{\diff}{\diff t}\int_t^{t+T}\norm{\hat P_s^{(\gamma)}f}_{L^2(\hat\mu)}^2\diff s=\norm{\hat P_{t+T}^{(\gamma)}f}_{L^2(\hat\mu)}^2-\norm{\hat P_{t}^{(\gamma)}f}_{L^2(\hat\mu)}^2\\
            &=\int_t^{t+T}\frac{\diff}{\diff s}\norm{\hat P_s^{(\gamma)}f}_{L^2(\hat\mu)}^2\diff s\\
            &=2\int_t^{t+T}\left(\hat P_s^{(\gamma)}f,(\hat L+\gamma(\Pi_v-I))\hat P_s^{(\gamma)}f\right)_{L^2(\hat\mu)}\diff s\\
            &\leq -2\gamma\int_t^{t+T}\left(\hat P_s^{(\gamma)}f,(I-\Pi_v)\hat P_s^{(\gamma)}f\right)_{L^2(\hat\mu)}\diff s\\
            &=-2\gamma\int_t^{t+T}\norm{(\Pi_v-I)\hat P_s^{(\gamma)}f}_{L^2(\hat\mu)}^2\diff s\\
            &=-2\gamma\norm{(\Pi_v-I)\hat P_{t+\cdot}^{(\gamma)}f}_{L^2(\muqk)}^2\,.
        \end{align*}
        Since $\hat P_\cdot^{(\gamma)} f\in\dom(A)$ and $A\hat P_\cdot^{(\gamma)} f=\gamma(\Pi_v-I)\hat P_\cdot^{(\gamma)} f$, the space-time Poincaré inequality yields
        \begin{equation*}
            \norm{\hat P_{t+\cdot}^{(\gamma)} f}_{L^2(\muqk)}^2\leq (\gamma^2C_0(T)+C_1(T))\norm{(\Pi_v-I)\hat P_{t+\cdot}^{(\gamma)} f}_{L^2(\muqk)}^2\,,
        \end{equation*}
        so that
        \begin{equation*}
            \frac{\diff}{\diff t}\int_t^{t+T}\norm{\hat P_s^{(\gamma)}f}_{L^2(\hat\mu)}^2\diff s\leq \frac{-2\gamma}{\gamma^2C_0(T)+C_1(T)}\int_t^{t+T}\norm{\hat P_s^{(\gamma)}f}_{L^2(\hat\mu)}^2\diff s\,.
        \end{equation*}
        Grönwall's inequality yields
        \begin{align*}
            \frac{1}{T}\int_t^{t+T}\norm{\hat P_s^{(\gamma)}f}_{L^2(\hat\mu)}^2\diff s&\leq\exp\left(\frac{-2\gamma}{\gamma^2C_0(T)+C_1(T)}t\right)\frac{1}{T}\int_0^T\norm{\hat P_s^{(\gamma)}f}_{L^2(\hat\mu)}^2\diff s\\
            &\leq e^{-2\nu t}\norm{f}_{L^2(\hat\mu)}^2\,.
        \end{align*}
        Finally, the Cauchy-Schwarz inequality yields
        \begin{equation*}
            \frac{1}{T}\int_t^{t+T}\norm{\hat P_s^{(\gamma)}f}_{L^2(\hat\mu)}\diff s\leq\left(\frac{1}{T}\int_t^{t+T}\norm{\hat P_s^{(\gamma)}f}_{L^2(\hat\mu)}^2\diff s\right)^\frac{1}{2}\leq e^{-\nu t}\norm{f}_{L^2(\hat\mu)}
        \end{equation*}
        and thus the $\nu$-exponential contractivity in $T$-average. The upper bound on $t_\rel(\eps)$ follows by \Cref{lem:epsrelaxation}.
    \end{proof}

    In the following we demonstrate how the framework of second-order lifts can be used to simplify the approach in \cite{Cao2019Langevin,Lu2022PDMP} to proving space-time Poincaré inequalities leading to quantitative hypocoercivity bounds.
    The main tools are a
    space-time property of lifts and
    a quantitative version of the divergence lemma.
    
\subsection{Space-time property of lifts}

    For $(t,x,v)\in [0,\infty )\times\R^d\times\R^d$ we set
    $ \pi (t,x,v)=(t,x)$, whereas $\pi (x,v)=x$.

    \begin{lemm}[Space-time property of lifts]\label{lem:xtproplift}
        Let $(\hat L,\dom(\hat L))$ be a lift of $(L,\dom(L))$. Then for any $f,g,h\in L^2(\muq)$ such that $f\circ\pi\in\dom(A)$ and $g(t,\cdot)\in\dom(L)$ for a.e.\ $t\in[0,T]$ with $Lg\in L^2(\muq)$, we have
        \begin{equation*}
            \frac{1}{2}\left(A(f\circ\pi),h\circ\pi+\hat L(g\circ\pi)\right)_{L^2(\muqk)}=-\frac{1}{2}(\partial_tf,h)_{L^2(\muq)}+\Ecal_T(f,g)\, ,
        \end{equation*}
        where $\Ecal_T(f,g)=\int_0^T\Ecal(f(t,\cdot),g(t,\cdot ))\diff t$ is the time-integrated Dirichlet form.
    \end{lemm}
    \begin{proof} 
    By Properties \eqref{eq:deflift0}, \eqref{eq:deflift1} and \eqref{eq:deflift2} in \Cref{def:lifts},
         \begin{align*} \left(\hat L (f\circ\pi),h\circ\pi\right)_{L^2(\muqk)}&=0\,,\qquad 
            \left(-\partial_t(f\circ\pi),\hat L(g\circ\pi)\right)_{L^2(\muqk)}=0\,,\quad\text{ and}\\
            \left(\hat L(f\circ\pi),\hat L(g\circ\pi)\right)_{L^2(\muqk)}&=\int_0^T\left(\hat L(f(t,\cdot )\circ\pi),\hat L(g(t,\cdot )\circ\pi)\right)_{L^2(\mu\otimes\kappa)}\diff t\\
            &=2\Ecal_T(f,g)\,.
        \end{align*}
       Moreover, $  \left(-\partial_t(f\circ\pi),h\circ\pi\right)_{L^2(\muqk)}= (-\partial_tf,h)_{L^2(\muq)} $.
    \end{proof}

    \begin{rema}
        If we set $h=-\partial_tg$, the space-time property reads
        \begin{equation*}
            \frac{1}{2}\big( A(f\circ\pi),A(g\circ\pi)\big)_{L^2(\muqk)}=\frac 12\big(\partial_tf,\partial_tg\big)_{L^2(\muq)}+\Ecal_T(f,g)\,.
        \end{equation*}
        For functions satisfying either periodic or Neumann boundary conditions
        on $[0,T]$, the right hand side is the Dirichlet form of
        the reversible Markov process with state space $[0,T]\times\R^d$ given 
        by Brownian motion in the first component and an
        independent process
        with generator $L$ in the other components. 
        The space-time property generalises the observation that the process with generator $A$ describing deterministic motion with speed $-1$ in the first (time) component and motion described by $\hat L$ in the other (space and velocity) components is a lift of this reversible process. 
    \end{rema}

\subsection{Divergence lemma}

    We now assume that for any $\gamma\geq0$, the semigroup $(\hat P_t^{(\gamma)})_{t\geq0}$ is a second-order lift of the transition semigroup $(P_t)_{t\geq0}$ acting on $L^2(\mu)$ of an overdamped Langevin diffusion with invariant probability measure $\mu(\diff x)\propto\exp(-U(x))\diff x$, see \Cref{ex:ODLifts}. Recall that the generator of $(P_t)_{t\geq0}$ is given by
    \begin{equation*}
        Lf=-\frac{1}{2}\nabla^*\nabla f=-\frac{1}{2}\nabla U\cdot\nabla f+\frac{1}{2}\Delta f
    \end{equation*}
    for $f\in\dom(L)$, where the adjoint is in $L^2(\mu)$. In particular,
    $$\mathcal E (f,g)=\frac{1}{2}\left( \nabla f,\nabla g\right)_{L^2(\mu )}\, .$$

    \begin{assu}\label{ass:U}\mbox{}
        \begin{enumerate}[(i)]
            \item The potential satisfies $U\in C^2(\R^d)$ and there exists a constant $\kappa_-\in [0,\infty )$ such that $\nabla^2 U(x)\succeq-\kappa_-I_d$ for all $x\in\R^d$.
            \item\label{ass:gap} The probability measure $\mu$ satisfies a Poincaré inequality with constant $m$, i.e. 
            \begin{equation*}
                \int_{\R^d}f^2\diff\mu\leq\frac{1}{m}\int_{\R^d}|\nabla f|^2\diff\mu\, \quad\text{ for all $f\in H^1(\mu)$ with $\int_{\R^d}f\diff\mu=0$}\,.
            \end{equation*}
            \item\label{ass:discretespec} The spectrum of $(L,\dom(L))$ on $L^2(\mu)$ is discrete.
        \end{enumerate}
    \end{assu}

    Assumption \ref{ass:U}\eqref{ass:discretespec}
    is for instance satisfied if
    $
        \lim_{|x|\to\infty}\frac{U(x)}{|x|^\alpha}=\infty
    $
    for some $\alpha>1$, cf.\ \cite{hooton1981compact}. 

    The next statement is a quantitative version of the divergence lemma for the domain $(0,T)\times\mathbb R^d$.
    The divergence lemma is equivalent to Lions' lemma, estimating the $L^2$-norm by the $H^{-1}$-norm of the gradient, see \cite{Amrouche2015Lions}. Let $\overline{\nabla}f $ denote the gradient in both time
    and space variables of a function $f$ defined on $(0,T)\times\mathbb R^d$, and let $\overline{\nabla}^*X$
    denote the adjoint divergence operator with respect to the
    measure $ \muq$ applied to a vector field 
    $X\colon(0,T)\times\R^d\to\R^{d+1}$.

    \begin{lemm}[Quantitative divergence lemma]\label{lem:divergence}
    Suppose that Assumption \ref{ass:U} is satisfied.
     Then for any $T\in (0,\infty )$, there exist constants $c_0(T),c_1(T)\in (0,\infty )$ such that for any $f\in L_0^2([0,T]\times\R^d ,\muq)$ there is
     a vector field $X\in H_0^{1,2}((0,T)\times\R^d\to\R^{d+1},\muq)$ with
        \begin{eqnarray}\label{boundX}
            \norm{X}_{L^2(\muq)}^2&=&\sum_{i=0}^d\norm{X_i}_{L^2(\muq)}^2\ \leq\ c_0(T)\norm{f}_{L^2(\muq)}^2\, ,\\
            \label{boundnablaX}
            \norm{\overline\nabla X}_{L^2(\muq)}^2&=&\sum_{i,j=0}^d\norm{\partial_iX_j}_{L^2(\muq)}^2\ \leq\ c_1(T)\norm{f}_{L^2(\muq)}^2\, ,
        \end{eqnarray} 
        such that
        \begin{equation}\label{fdivX}
            f\ =\ \overline{\nabla}^*X\,  . 
        \end{equation}
      The constants can be chosen as
        \begin{eqnarray}
            c_0(T)&=&19T^2+70\frac{1}{m}\, ,\label{eq:c0}\\
            c_1(T)&=&328+75\kappa_-\max\left(\frac{1}{\sqrt{m}},\frac{T}{\pi}\right)^2+\frac{1821}{mT^2}\, ,\label{eq:c1}
        \end{eqnarray}           
        and the vector field can be represented as  
        $$X=\begin{pmatrix}-h\\\nabla_xg\end{pmatrix}\
        \text{with functions }h\in H_0^{1,2}((0,T)\times\R^d,\muq )\text{ and }g\in H_0^{2,2}((0,T)\times\R^d,\muq)\, .$$
    \end{lemm}
    
     A corresponding result has first been proven in  \cite{Cao2019Langevin} with constants that are not of the optimal order, see also \cite{BrigatiStoltz2023Decay}. 
     An improved bound with constants of the order stated above 
     has been conjectured in \cite{Cao2019Langevin}. Based on
     a careful extension of the arguments in this paper,
     the precise result above is proven in a work in preparation
     \cite{EberleLoerler2024}. 

     \begin{rema}\label{rem:choiceofT}
         The bounds resulting from the divergence lemma considered below are optimised when choosing $T$ proportional to 
         $1/\sqrt{m}$. For $T=3/\sqrt{m}$, \eqref{eq:c0} and \eqref{eq:c1} yield
         \begin{equation}\label{eq:c_0c_1_opt}
             c_0(T)= \frac{241}{m}\quad\text{ and }\quad c_1(T)=530\tfrac{1}{3}+75\frac{\kappa_-}{m} \, .
         \end{equation}
     \end{rema}

     \subsection{Space-time Poincaré inequalities via lifts}
    
    For simplicity and clarity of the exposition we now assume as in \Cref{ex:ODLifts} that $\hat L$ is the generator of the deterministic Hamiltonian flow. We stress, however, that with appropriate modifications, the approach is much more generally applicable.
    Our goal is to show how the space-time property of lifts combined with the divergence lemma can be used to prove the space-time Poincaré inequality stated in \Cref{thm:stpi} below.\medskip
    
    We consider the adjoint $(A^*,\dom(A^*))$ in $L^2([0,T]\times\R^d\times\R^d,\muqk)$ of the operator $(A,\dom(A))$
    defined in \eqref{def:A}. By integration by parts and the antisymmetry of $\hat L$,
    \begin{align*}
        (Af,g)_{L^2(\muqk)} &= (f,-Ag)_{L^2(\muqk)} - \int_{\R^{2d}}(f(T,\cdot)g(T,\cdot)-f(0,\cdot)g(0,\cdot))\diff\hat\mu
    \end{align*}
    for all $f,g\in\dom(A)$. Thus functions
    $g\in\dom(A)$ satisfying $g(0,\cdot)=g(T,\cdot)=0$ are contained in the domain of $A^*$, and
    \begin{equation}\label{Astarg}
        A^*g=-Ag=\partial_tg-\hat Lg\, . 
    \end{equation}
    For $v\in\R^d$ we set $\overline{v}=(-1,v)\in\R^{d+1}$,
    so that $A=\overline{v}\cdot\overline{\nabla}-\nabla U(x)\cdot\nabla_v$.
    
    The space-time property of lifts yields the following.
    
    \begin{lemm}
        \label{lem:Astar}
        Let $X=\begin{pmatrix}-h\\\nabla g\end{pmatrix}$ with $h\in H_0^{1,2}(\muq)$ and $g\in H_0^{2,2}(\muq)$. Then 
        \begin{eqnarray}\label{eq:Astar1}
            \int_{\R^d} A^*(\overline{v}\cdot X\circ\pi)\kappa(\diff v) &=& \overline{\nabla}^*X\circ\pi \, ,\quad\text{ and}\\
           \norm{A^*(\overline{v}\cdot X\circ\pi)-\overline\nabla^*X\circ\pi}_{L^2(\muqk)}
            &\leq &\sqrt{2}\norm{\overline{\nabla}X}_{L^2(\muq)}\,.\label{eq:Astar2}
        \end{eqnarray}
    \end{lemm} 
    
    \begin{proof}
        The Dirichlet boundary values in time of $X$ give $\overline{v}\cdot X\circ\pi\in\dom(A^*)$. Therefore, by \Cref{lem:xtproplift}, for any function $\varphi\in C_0^\infty ((0,T)\times\R^d)$,
        \begin{eqnarray*}
         \lefteqn{\left(\varphi\circ\pi,\overline{\nabla}^*X\circ\pi\right)_{L^2(\muqk)}\ =\
            \left(\varphi,\overline{\nabla}^*X\right)_{L^2(\muq)}\ =\ (\overline{\nabla}\varphi,X)_{L^2(\muq)}}\\
            &=&-(\partial_t\varphi,h)_{L^2(\muq)}+2\Ecal_T(\varphi,g)\
            =\ \left(A(\varphi\circ\pi),h\circ\pi+\hat L( g\circ\pi)\right)_{L^2(\muqk)}\\
            &=&\left(A(\varphi\circ\pi),\overline{v}\cdot X\circ\pi\right)_{L^2(\muqk)}\
        =\ \left(\varphi\circ\pi,A^*(\overline{v}\cdot X\circ\pi)\right)_{L^2(\muqk)} .
        \end{eqnarray*}
        This proves \eqref{eq:Astar1}.
        For the proof of \eqref{eq:Astar2}
        note that since $A^*=-\overline{v}\cdot\overline{\nabla}+\nabla U(x)\cdot\nabla_v$,
        \begin{equation}\label{eq:Astar3}
            \norm{A^*(\overline{v}\cdot X\circ\pi)-\overline\nabla^*X\circ\pi}_{L^2(\muqk)}
            = \norm{\overline{v}\cdot\overline{\nabla}(\overline{v}\cdot X\circ\pi)-\overline{\nabla}\cdot X\circ\pi}_{L^2(\muqk)}\,. 
        \end{equation}
            Moreover, noting that  $\int_{\R^d} v_i^2v_k^2\diff\kappa=\begin{cases}
                3&\textup{if }i=k>0,\\
                1&\textup{otherwise,}
            \end{cases}$
            we obtain
        \begin{align*}
            \MoveEqLeft\norm{\overline{v}\cdot\overline{\nabla}(\overline{v}\cdot X\circ\pi)}_{L^2(\muqk)}^2=\sum_{i,j,k,l=0}^d(\partial_iX_j,\partial_kX_l)_{L^2(\muq)}\int_{\R^d} v_iv_jv_kv_l\diff\kappa\\
            &=\sum_{i,k=0}^d(\partial_iX_i,\partial_kX_k)_{L^2(\muq)}\int v_i^2v_k^2\diff\kappa+\sum_{\substack{i,j=0\\i\neq j}}^d\left((\partial_iX_j,\partial_iX_j)_{L^2(\muq)}+(\partial_iX_j,\partial_jX_i)_{L^2(\muq)}\right)\\
            &=\sum_{i,j=0}^d\left((\partial_iX_i,\partial_jX_j)_{L^2(\muq)}+(\partial_iX_j,\partial_iX_j)_{L^2(\muq)}+(\partial_iX_j,\partial_jX_i)_{L^2(\muq)}\right)-2\norm{\partial_0X_0}_{L^2(\muq)}^2\\
            &\leq \norm{\overline\nabla\cdot X}_{L^2(\muq)}^2+2\norm{\overline{\nabla}X}_{L^2(\muq)}^2\, .
        \end{align*}
        Since $\int_{\R^d}\overline{v}\cdot\overline{\nabla}(\overline{v}\cdot X\circ\pi)\diff\kappa=\overline{\nabla}\cdot X\circ\pi$ we conclude that
        \begin{equation*}
            \norm{\overline{v}\cdot\overline{\nabla}(\overline{v}\cdot X\circ\pi)-\overline{\nabla}\cdot X\circ\pi}_{L^2(\muqk)}^2=\norm{\overline{v}\cdot\overline{\nabla}(\overline{v}\cdot X\circ\pi)}_{L^2(\muqk)}^2-\norm{\overline{\nabla}\cdot X}_{L^2(\muq)}^2\leq 2\norm{\overline{\nabla}X}_{L^2(\muq)}^2\,.
        \end{equation*}
        By \eqref{eq:Astar3}, this implies \eqref{eq:Astar2}.
    \end{proof}

    \begin{theo}[Space-time Poincaré inequality] \label{thm:stpi}
        Suppose that $\mu$ satisfies \Cref{ass:U}. Then for any $T\in (0,\infty )$, the generator $\hat L$ of the Hamiltonian flow satisfies the inequality $\stpi\left(T,2c_0(T),3+4c_1(T)\right)$ with the constants $c_0(T)$ and $c_1(T)$ from Lemma \ref{lem:divergence}.
    \end{theo}
    
   The following proof is an adapted, optimised, and partially simplified version of the proof of \cite[Theorem 2.1]{Lu2022PDMP}.
    
    \begin{proof}
        By Remark \ref{rem:core}, it is sufficient to verify
        the space-time Poincaré inequality for functions 
        $f_0\in C_0^\infty([0,T]\times\R^d\times\R^d)$.
        Hence fix such a function, let $f=f_0-\int f_0\diff(\muqk)$, and let $\tilde f\colon[0,T]\times\R^d\to\R$ be such that $\Pi_vf=\tilde f\circ\pi$.
        First note that 
        \begin{equation}\label{L2decompf}
         \norm{f}_{L^2(\muqk)}^2=\norm{f-\Pi_vf}_{L^2(\muqk)}^2+\norm{\tilde f}_{L^2(\muq)}^2\,, 
        \end{equation}
         so that it suffices to bound $\norm{\tilde f}_{L^2(\muq)}^2$ by $\norm{f-\Pi_vf}_{L^2(\muqk)}^2$ and $\norm{Af}_{L^2(\muqk)}^2$. 
        By \Cref{lem:divergence} applied to $\tilde f$, there
        exist functions $h\in H_0^{1,2}((0,T)\times\R^d,\muq)$ and $g\in H_0^{2,2}((0,T)\times\R^d,\muq)$ such that the vector field $X=\begin{pmatrix}-h\\\nabla g\end{pmatrix}$ satisfies
        \eqref{boundX}, \eqref{boundnablaX} and \eqref{fdivX} with respect to $\tilde f$. In particular,
        \begin{equation}\label{L2tildef}
            \norm{\tilde f}_{L^2(\muq)}^2=(\tilde f,\tilde f)_{L^2(\muq)}=(\tilde f,\overline\nabla^*X)_{L^2(\muq)}=(\overline{\nabla}\tilde f,X)_{L^2(\muq)}.
        \end{equation}
        Using the space-time property of lifts (\Cref{lem:xtproplift}), we see that
        \begin{align*}
            (\overline{\nabla}\tilde f,X)_{L^2(\muq)}&=\big(-\partial_t\tilde f,h\big)_{L^2(\muq)}+2\Ecal_T(\tilde f,g)\\
            &=\left(A(\tilde f\circ\pi),h\circ\pi+\hat L(g\circ\pi)\right)_{L^2(\muqk)}\\
            &=\left(A(\tilde f\circ\pi),\overline{v}\cdot X\circ\pi\right)_{L^2(\muqk)}\\
            &=\big(Af,\overline{v}\cdot X\circ\pi\big)_{L^2(\muqk)}+\big(A(\Pi_vf-f),\overline{v}\cdot X\circ\pi\big)_{L^2(\muqk)}\,.
        \end{align*}
        Let us consider the two summands separately.
    
        The first summand can be estimated as
        \begin{align*}
            \left(Af,\overline{v}\cdot X\circ\pi\right)_{L^2(\muqk)}&\leq \norm{Af}_{L^2(\muqk)}\norm{\overline{v}\cdot X\circ\pi}_{L^2(\muqk)}=\norm{Af}_{L^2(\muqk)}\norm{X}_{L^2(\muq)}\\
            &\leq c_0(T)^\frac{1}{2}\norm{Af}_{L^2(\muqk)}\norm{\tilde f}_{L^2(\muq)}\, ,
        \end{align*}
        where we used $\int_{\R^d} v_iv_j\diff\kappa=\delta_{ij}$ and $c_0(T)$ is the constant from \Cref{lem:divergence}.
    
        The second summand satisfies
        \begin{equation*}
            \left(A(\Pi_vf-f),\overline{v}\cdot X\circ\pi\right)_{L^2(\muqk)}=\left(\Pi_vf-f,A^*(\overline{v}\cdot X\circ\pi)\right)_{L^2(\muqk)},
        \end{equation*}
        where $A^*$ is the adjoint of $A$ in $L^2(\muqk)$. 
        Here we have used that by the conditions on $h$ and $g$, $\overline{v}\cdot X\circ\pi = h\circ\pi+\hat L(g\circ\pi)$
        is in the domain of $A$, and hence by \eqref{Astarg} in the 
        domain of $A^*$. Since $\tilde f= \overline\nabla^*X$,
        \Cref{lem:Astar} and \Cref{lem:divergence} yield
        \begin{eqnarray*}
            \norm{A^*(\overline{v}\cdot X\circ\pi)}_{L^2(\muqk)}^2
            &=&\norm{\tilde f}_{L^2(\muq)}^2+\norm{A^*(\overline{v}\cdot X\circ\pi)-\overline\nabla^*X\circ\pi}_{L^2(\muqk)}^2\,\\
            &\leq &\norm{\tilde f}_{L^2(\muq)}^2+2\norm{\overline\nabla X}_{L^2(\muq)}^2\leq(1+2c_1(T))\norm{\tilde f}_{L^2(\muq)}^2\,,
        \end{eqnarray*}
        so that the second summand can be estimated as
        \begin{equation*}
            \left(A(\Pi_vf-f),\overline{v}\cdot X\circ\pi\right)_{L^2(\muqk)}\leq \left(1+2c_1(T)\right)^\frac{1}{2}\norm{f-\Pi_vf}_{L^2(\muqk)}\norm{\tilde f}_{L^2(\muqk)}\,.
        \end{equation*}
    
        Putting things together, we hence obtain by \eqref{L2tildef},
        \begin{align*}
            \norm{\tilde f}_{L^2(\muq)}&\leq c_0(T)^\frac{1}{2}\norm{Af}_{L^2(\muqk)}+\left(1+2c_1(T)\right)^\frac{1}{2}\norm{f-\Pi_vf}_{L^2(\muqk)}\,,
        \end{align*}
        so that by \eqref{L2decompf},
        \begin{align*}
            \norm{f}_{L^2(\muqk)}^2
            &\leq2c_0(T)\norm{Af}_{L^2(\muqk)}^2+(3+4c_1(T))\norm{f-\Pi_vf}_{L^2(\muqk)}^2\,.\qedhere
        \end{align*}
        
    \end{proof}

    \begin{rema}
        In order to generalise the above proof technique to other second-order lifts of an overdamped Langevin diffusion, the key step is bounding $\norm{A^*(\overline{v}\cdot X\circ\pi)}_{L^2(\muqk)}$ by $\norm{\tilde f}_{L^2(\muq)}$. Further care regarding the domain of the operator $A^*$ must also be taken in more general settings.
    \end{rema}

\subsection{Optimality of randomised Hamiltonian Monte Carlo}

By \Cref{thm:decay}, the space-time Poincaré inequality for the generator $\hat L$ of the Hamiltonian flow immediately gives the following quantitative hypocoercivity result for the transition semigroup $(\hat P_t^{(\gamma)})_{t\geq0}$ of randomised Hamiltonian Monte Carlo with generator
\begin{equation*}
    \hat L^{(\gamma)} = \hat L + \gamma(\Pi_v-I)\, .
\end{equation*}
        
    \begin{coro}[Exponential contractivity of randomised HMC]\label{cor:RHMCcontractivity}
        For any $\gamma ,T\in (0,\infty )$, the semigroup $(\hat P_t^{(\gamma)})_{t\geq0}$ is exponentially contractive in $T$-average with rate
        \begin{equation}\label{eq:nuc0c1}
            \nu=\frac{\gamma}{2c_0(T)\gamma^2+3+4c_1(T)}\, ,
        \end{equation}
        where $c_0(T)$ and $c_1(T)$ are the constants from \Cref{lem:divergence}.
    \end{coro}

    \begin{proof}
        The expression for $\nu$ follows immediately from \Cref{thm:decay,thm:stpi}.
    \end{proof}

    \begin{rema}[Optimal refresh rate]
        For fixed $T\in (0,\infty )$, the rate $\nu$ in \eqref{eq:nuc0c1} is maximal when $\gamma$ is chosen such that
        the two summands in the denominator coincide, i.e.\ for
        $\gamma =\sqrt{(3+4c_1(T))/(2c_0(T))}$. In this case,
        \begin{equation}\label{nuinverseopt}
            \frac 1\nu = \sqrt{8c_0(T)(3+4c_1(T))}\, .
        \end{equation}
    \end{rema}
    To obtain explicit bounds, we set $T= {3}/{\sqrt{m}}$.  
    By \eqref{eq:nuc0c1} and Remark \ref{rem:choiceofT}, this yields
    \begin{equation}\label{eq:nuc0c1opt}
            \frac{1}{\nu}=\frac{482}{\sqrt m}\left( \frac{\gamma}{\sqrt m}+\left(2124\tfrac{1}{3}+300 \frac{\kappa_-}{m}\right)\frac{\sqrt{m}}{482\gamma}\right)
        <\frac{482}{\sqrt m}\left( \frac{\gamma}{\sqrt m}+5\left(1+ \frac{\kappa_-}{7m}\right){\frac{\sqrt{m}}{\gamma}}\right)\,.
    \end{equation}
    The decay rate $\nu$ is maximal for
    \begin{equation}\label{gammaoptimal}
        \gamma = \sqrt{(2124\tfrac{1}{3}\, m+300\, \kappa_-)/482}\,,
    \end{equation}
    and in this case, 
    \begin{equation}\label{eq:nuoptimalgamma}
        \frac{1}{\nu} = 2\cdot\frac{482\gamma}{m} < \frac{2024}{\sqrt m}\cdot \sqrt{ 1+\frac{\kappa_-}{7m} }\, .
    \end{equation}

    \begin{coro}[$C$-optimality of randomised HMC]\label{cor:optHMC}
        Let $c\in[0,\infty)$. On the class of all potentials satisfying \Cref{ass:U} with $\kappa_-\leq cm$, randomised Hamiltonian Monte Carlo with $\gamma$ given by \eqref{gammaoptimal} is a $C$-optimal lift of the corresponding overdamped Langevin diffusion with 
        $$C=2\sqrt{2}(2024\sqrt{1+c/7}+3)\, .$$
        More generally, for $A\in [1,\infty )$, randomised HMC
        is a $C(A)$-optimal lift with
        $$ C(A)= 2\sqrt 2 \left(482\cdot (6+5c/7)A+3\right)$$
        for any $\gamma$ satisfying $1/A\le \frac{\gamma}{\sqrt{m}}\le A$.
    \end{coro}
    \begin{proof} For $\gamma $ given by \eqref{gammaoptimal}, 
    the bound in \eqref{eq:nuoptimalgamma} yields
        \begin{equation*}
            t_\rel(\hat P^{(\gamma)}) \leq T+\nu^{-1}\leq \frac{1}{\sqrt{m}}(2024\sqrt{1+c/7}+3)\,,
        \end{equation*}
        whereas the relaxation time of the underlying overdamped Langevin diffusion is $t_\rel(P)=\frac{1}{m}$. This proves
        the first claim, and the second claim follows similarly 
        from \eqref{eq:nuc0c1opt}.
    \end{proof}

    \begin{rema}[Optimal lifts with and without convexity]
        The explicit values of the constants in Corollary \ref{cor:optHMC} are certainly not optimal. In particular, we expect that the constants \eqref{eq:c_0c_1_opt} in the quantitative divergence lemma can still be improved substantially. Therefore, Corollary \ref{cor:optHMC} seems 
        to indicate that in the convex or mildly non-convex case, randomised HMC with optimised choice of $\gamma$ is not too
        far from being an optimal lift of overdamped Langevin.
        On the other hand, a corresponding statement cannot
        be true in very non-convex cases. For example, for a 
        double well-potential $U(x)=\beta (x^2-1)^2$, by the Arrhenius law, the relaxation times of overdamped Langevin and randomised
        HMC are both expected to be asymptotically of exponential order $\exp (\beta (U(0)-U(1)) $ as $\beta\to\infty$.
        Therefore, a square root acceleration cannot take place
        for large $\beta$, and the dependence of the bounds on
        the non-convexity parameter $c$ is natural.
    \end{rema}

    \begin{rema}[Hypocoercivity approaches]\label{rem:DMS}
        Based on the work of Hérau \cite{Herau2006Boltzmann}, Dolbeault, Mouhot and Schmeiser \cite{DMS2009Hypocoercivity,DMS2015Hypocoercivity} developed a method for proving hypocoercivity for a general class of linear kinetic equations, covering the kinetic Fokker-Planck and the linear Boltzmann equations and thereby Langevin dynamics and randomised HMC. Their method is based on a modification of the $L^2(\muk)$-norm by a suitable perturbation which allows to recover dissipation in all variables. 
        As pointed out to us by Lihan Wang,
        if $\hat L$ is the generator of the Hamiltonian flow, then Assumption H3 in \cite{DMS2015Hypocoercivity} resembles our first-order condition \eqref{eq:deflift1} for a second-order lift, while Assumption H2 in \cite{DMS2015Hypocoercivity} is implied by the second-order condition \eqref{eq:deflift2} together with \Cref{ass:U}\eqref{ass:gap} on the spectral gap of the underlying diffusion.
        Since the entropy functional in \cite{DMS2015Hypocoercivity} does not directly capture the dynamics of the process, one cannot expect optimal bounds on the convergence rate. Indeed, as calculated in \cite[Appendix B]{Cao2019Langevin}, assuming a lower bound on the Hessian of the potential and in the case that $\gamma$ is proportional to $\sqrt{m}$, the approach in \cite{DMS2015Hypocoercivity} seems to yield a rate of order $m^{5/2}$ in the limit as $m\to 0$ instead of the optimal order $m^{1/2}$. 
    \end{rema}

\section*{Statements and Declarations}

    \noindent\textbf{Funding.}\hspace{2ex}
    Gef\"ordert durch die Deutsche Forschungsgemeinschaft (DFG) im Rahmen der Exzellenzstrategie des Bundes und der L\"ander -- GZ2047/1, Projekt-ID 390685813.\\
    The authors were funded by the Deutsche Forschungsgemeinschaft (DFG, German Research Foundation) under Germany's Excellence Strategy  -- GZ 2047/1, Project-ID 390685813.\smallskip\\
    \textbf{Acknowledgements.}\hspace{2ex}
    The authors would like to thank Lihan Wang, Arnaud Guillin and L\'eo Hahn for many helpful discussions.

\printbibliography

\end{document}